\documentclass[11pt]{article}
\usepackage{amsmath,amssymb,amsfonts,latexsym,xspace,amsthm,array,float}
\usepackage[utf8]{inputenc}
\usepackage{a4wide,hyperref}
\usepackage[ruled,vlined]{algorithm2e}
\usepackage{tikz}
\usetikzlibrary{patterns,shapes,fit}
\usepackage{enumerate}

\tikzstyle{vert}=[green]
\newcommand{\conf}{c}

\newcommand{\stackcs}{stack configurations\xspace}
\newcommand{\stackw}{stack word\xspace}
\newcommand{\stackws}{stack words\xspace}
\newcommand{\sortingw}{sorting word\xspace}
\newcommand{\sortingws}{sorting words\xspace}
\newcommand{\sortable}{$2$-stack sortable\xspace}
\newcommand{\rlm}{\ensuremath{\rho, \lambda, \mu}\xspace}
\newcommand{\rlmi}{\ensuremath{\rho_i, \lambda_i, \mu_i}\xspace}
\newcommand{\dec}[2]{\ensuremath{\hat{#1}^#2}}

\newcommand{\G}[1]{\ensuremath{{\mathcal G}^{(#1)}}}
\newtheorem{thm}{Theorem}

\newtheorem{defn}[thm]{Definition}
\newtheorem{lem}[thm]{Lemma}

\newcommand{\pushall}{ $2$-stack pushall sortable }
\newcommand{\gLabel}[1]{ \tikz[baseline=-3pt]\node[inner sep=1pt,circle,white,fill=black]{$#1$};}
\newcommand{\gConf}[2]{\ensuremath{ (#1,} \tikz[baseline=-3pt]\node[inner sep=1pt,circle,white,fill=black]{$#2$}; \ensuremath{)} }

\newcounter{indiceD}
\newcounter{tailleH}
\newcounter{tailleV}
\newcounter{nombrenoeud}
\newcommand{\ssConf}[2]{
\setcounter{tailleH}{0};
\setcounter{tailleV}{0};
\foreach \i in { #1 } {
\addtocounter{tailleH}{1};
	}
\foreach \i in {#2}{
\addtocounter{tailleV}{1};
	}
\draw (1.5, \thetailleH+.5) -- (1.5,0) -- (2.5,0) -- (2.5,\thetailleH+.5);
\draw (0, \thetailleV+.5) -- (0,0) -- (1,0) -- (1,\thetailleV+.5);
\setcounter{indiceD}{0}
\foreach \i in { #1 } {
\draw (2, \theindiceD+.5) node {\i};
\addtocounter{indiceD}{1};
}
\setcounter{indiceD}{0}
\foreach \i in { #2 } {
\draw (.5, \theindiceD+.5) node {\i};
\addtocounter{indiceD}{1};
}
}

\newcommand{\ssConfIR}[4]{
\setcounter{tailleH}{0};
\setcounter{tailleV}{0};
\foreach \i in { #1 } {
\addtocounter{tailleH}{1};
	}
\foreach \i in {#2}{
\addtocounter{tailleV}{1};
	}
\draw (1.5, \thetailleH+.5) -- (1.5,0) -- (2.5,0) -- (2.5,\thetailleH+.5);
\draw (0, \thetailleV+.5) -- (0,0) -- (1,0) -- (1,\thetailleV+.5);
\setcounter{indiceD}{0}
\foreach \i in { #1 } {
\draw (2, \theindiceD+.5) node {\i};
\addtocounter{indiceD}{1};
}
\setcounter{indiceD}{0}
\foreach \i in { #2 } {
\draw (.5, \theindiceD+.5) node {\i};
\addtocounter{indiceD}{1};
}
\draw (3,1) node[right] {${#3}$};
\draw (-0.5,1) node[left] {${#4}$};
}

\newcommand{\ssConfLab}[4]{
\setcounter{tailleH}{0};
\setcounter{tailleV}{0};
\foreach \i in { #1 } {
\addtocounter{tailleH}{1};
	}
\foreach \i in {#2}{
\addtocounter{tailleV}{1};
	}
\draw (0,0) node (#3) [draw,ellipse,matrix,inner sep=1pt] {
\scriptsize
\begin{scope}[scale=.3]
\draw (2.5,0) node {};
\draw (2.5,\thetailleH+.5) node  {};
\draw (0,0) node {};
\draw (0,\thetailleV+.5) node {};
 
\draw (1.5, \thetailleH+.5) -- (1.5,0) -- (2.5,0) -- (2.5,\thetailleH+.5);
\draw (0, \thetailleV+.5) -- (0,0) -- (1,0) -- (1,\thetailleV+.5);
\setcounter{indiceD}{0}
\foreach \i in { #1 } {
\draw (2, \theindiceD+.5) node {\i};
\addtocounter{indiceD}{1};
}
\setcounter{indiceD}{0}
\foreach \i in { #2 } {
\draw (.5, \theindiceD+.5) node {\i};
\addtocounter{indiceD}{1};
}
\draw (3.2,0.4) node[white, fill=black] {$#4$};
\end{scope}
\\};
}

\title{2-Stack  Sorting is polynomial \footnote{This work was completed with the support of the ANR
   project ANR BLAN-0204\_07  MAGNUM}}
\author{Adeline Pierrot \and Dominique Rossin}

\begin{document}
\maketitle

In this article, we give a polynomial algorithm to decide whether a given permutation $\sigma$ is sortable with two stacks in series. 
This  is indeed a longstanding open problem which was first introduced by Knuth in \cite{Knuth73}. 
He introduced the stack sorting problem as well as permutation patterns which arises naturally when characterizing permutations that can be sorted with one stack. 
When several stacks in series are considered, few results are known. 
There are two main different problems. 
The first one is the complexity of deciding if a permutation is sortable or not, the second one being the characterization and the enumeration of those sortable permutations.  
We hereby prove that the first problem lies in P by giving a polynomial algorithm to solve it.
This article strongly relies on \cite{PR13} in which $2$-stack pushall sorting is defined and studied.

%\paragraph{Warning:}
%The present article is a preliminary version of our work.
%The redaction is not final yet and it might be hard to read.
%We will update this article soon with a new improved and clearer version.
%In the meantime, we suggest the reader to first read \cite{PR13}, which
%is a requirement for the present article.
%\paragraph{Warning:}
%The present article is a preliminary version of our work.
%The redaction is not final yet and it might be hard to read.
%We will update this article soon with a new improved and clearer version.
%In the meantime, we suggest the reader to first read \cite{PR13}, which
%is a requirement for the present article.

\section{Notations and definitions}

Let $I$ be a set of integers. A permutation of $I$ is a bijection from $I$ onto $I$.
We write a permutation $\sigma$ of $I$ as the word $\sigma = \sigma_{1} \sigma_{2} \ldots \sigma_{n}$ where $\sigma_i = \sigma(i_1)$ with $I=\{i_1 \dots i_n\}$ and $i_1<i_2< \dots <i_n$.
The size of the permutation is the integer $n$ and if not precised, $I = [1 \dots n]$.
% For any integers $i$ and $j$, $[i \dots j] = \{k \mid i<k<j\}$
Notice that given the word $\sigma_{1} \sigma_{2} \ldots \sigma_{n}$ we can deduce the set $I$ and the map $\sigma$.
For any subset $J$ of $I$, $\sigma_{|J}$ denotes the permutation obtained by restricting $\sigma$ to $J$.
In particular the word corresponding to $\sigma_{|J}$ is a subword of the word corresponding to $\sigma$.

Let's recall the problem of sorting with two stacks in series. 
Given two stacks $H$ and $V$ in series --as shown in Figure \ref{fig:rholambdamu}-- and a permutation $\sigma$, we want to sort elements of $\sigma$ using the stacks. 
We write $\sigma$ as the word $\sigma = \sigma_{1} \sigma_{2} \ldots \sigma_{n}$ with $\sigma_{i} = \sigma(i)$ and take it as input. 
Then we have three different operations:
\begin{itemize}
\item $\rho$ which consist in pushing the next element of $\sigma$ on the top of $H$.
\item $\lambda$ which transfer the topmost element of $H$ on the top of $V$.
\item $\mu$ which pop the topmost element of $V$ and write it in the output.
\end{itemize}

\begin{figure}[H]
\begin{center}
\begin{tikzpicture}
\draw[very thick] (0,2) -- (0,0) -- (1,0) -- (1,2);
\draw (2.5,-0.5) node {$H$};
\draw[very thick] (2,2) -- (2,0) -- (3,0) -- (3,2);
\draw (0.5,-0.5) node {$V$};
\draw[very thick, dashed, ->] (4,2.5) -- node[above] {$\rho$} (2.66,2.5) -- (2.66,1.7);
\draw (4.8,2.5) node {{\small INPUT}};
\draw[very thick,dashed, ->] (2.33,1.7) -- (2.33,2.5) -- node[above] {$\lambda$} (0.66,2.5) -- (0.66,1.7);
\draw[very thick, dashed, <-] (-1,2.5) -- node[above] {$\mu$} (0.33,2.5) -- (0.33,1.7);
\draw (-2.5,2.5) node {{\small OUTPUT}};
\end{tikzpicture}
\caption{Sorting with two stacks in serie \label{fig:rholambdamu}}
\end{center}
\end{figure}

If there exists a sequence $w = w_{1} \ldots w_{k}$ of operations \rlm that leads to the identity in the output, we say that the permutation $\sigma$ is \sortable. 
In that case, we define the sorting word associated to this sorting process as the word $w$ on the alphabet $\{\rlm\}$. 
Notice that necessarily $w$ has $n$ times each letter $\rho, \lambda$ and $\mu$ and $k = 3n$. 
For example, permutation $2431$ is sortable using the following process. 

\vspace{0.2 cm}

{\scriptsize
\begin{tabular}{|c|c|c|c|c|c|c|}
\begin{tikzpicture}[scale=.3]
\ssConfIR{2}{}{4\,3\,1}{}
\end{tikzpicture}&
\begin{tikzpicture}[scale=.3]
\ssConfIR{2,4}{}{3\,1}{}
\end{tikzpicture}&
\begin{tikzpicture}[scale=.3]
\ssConfIR{2}{4}{3\,1}{}
\end{tikzpicture}&
\begin{tikzpicture}[scale=.3]
\ssConfIR{2,3}{4}{1}{}
\end{tikzpicture}&
\begin{tikzpicture}[scale=.3]
\ssConfIR{2}{4,3}{1}{}
\end{tikzpicture}&
\begin{tikzpicture}[scale=.3]
\ssConfIR{2,1}{4,3}{}{}
\end{tikzpicture}\\
\hline
\begin{tikzpicture}[scale=.3]
\ssConfIR{2}{4,3,1}{}{}
\end{tikzpicture}&
\begin{tikzpicture}[scale=.3]
\ssConfIR{2}{4,3}{}{1}
\end{tikzpicture}&
\begin{tikzpicture}[scale=.3]
\ssConfIR{}{4,3,2}{}{1}
\end{tikzpicture}&
\begin{tikzpicture}[scale=.3]
\ssConfIR{}{4,3}{}{1\,2}
\end{tikzpicture}&
\begin{tikzpicture}[scale=.3]
\ssConfIR{}{4}{}{1\,2\,3}
\end{tikzpicture}&
\begin{tikzpicture}[scale=.3]
\ssConfIR{}{}{}{1\,2\,3\,4}
\end{tikzpicture}
\end{tabular}
}
%Push $2$ onto $H$. Push $4$ onto $H$, then $V$. Push $3$ onto $H$ then $V$. Push $1$ onto $H$ then $V$ and finally output it. 
%Output $2$, then $3$ then $4$.

\vspace{0.2 cm}

This sorting is encoded by the word $w = \rho \rho \lambda \rho \lambda \rho \lambda \mu \lambda \mu \mu \mu$. 
We can also decorate the word to specify the element on which the operation is performed. 
The {\em decorated word} for $w$ and $2431$ is 
$\hat{w} =  \rho_{2} \rho_{4} \lambda_{4} \rho_{3} \lambda_{3} \rho_{1} \lambda_{1} \mu_{1} \lambda_{2} \mu_{2} \mu_{3} \mu_{4}$. 
Note that we have the same information between $(\sigma, w)$ and $\hat{w}$. 
Nevertheless, in a decorated word appears only once each letter $\rho_{i}, \lambda_{i}$ or $\mu_{i}$.
The decorated word associated to $(\sigma, w)$ is denoted $\hat{w}^\sigma$.
% \marginpar{exactly once? Attention def 4}
% TODO : décider si les decorated words doivent venir de sorting words

Of course not all permutations are sortable. 
The smallest non-sortable ones are of size $7$, for instance $\sigma=2 4 3 5 7 6 1$.

When only one stack is considered, there exists a natural algorithm to decide whether a permutation is sortable or not. 
Indeed, there is a unique way to sort a permutation using only one stack, and a greedy algorithm gives a decision procedure. 
For two stacks in series, a permutation can be sorted in numerous ways. 
Take for example permutation $4321$. 
Each element can be pushed in either stacks $H$ or $V$ and output the identity at the end. 
Thus the decreasing permutation of size $n$ has more than $2^{n}$ ways to be sorted {\em i.e.} more than $2^{n}$ sorting words.

%First, we prove that among all sortings of a permutation, there exist sortings which respect some easy properties. 
%For example, we want to output elements {\it as soon as possible}. 

Several articles introduce restrictions either on the rules or on the stack structure. 
For example, in his PhD-thesis West introduced a greedy model with decreasing stacks \cite{West93}.
Permutations sortable with this model, called West-$2$-stack sortable permutations, are characterized and enumerated.

%In fact, for two-stack sorting, finding a canonical sorting is hard. Thus many slightly modified models appear.
%For example, West showed that a canonical algorithm exists if elements in the stacks are always in decreasing order.

For our unrestricted case called sometimes in litterature {\em general 2-stack sorting problem}, 
no characterization of sortable permutations and no polynomial algorithm to decide if a permutation is sortable is known.
A common mistake when trying to sort a given permutation is to pop out the smallest element $i$ as soon as it lies in the stacks.
This operation may indeed move other elements if $i$ is not the topmost element of $H$. 
The elements above it are then transferred into $V$ before $i$ can be popped out. 
But sometimes, it can be necessary to take some elements of $\sigma$ and push them onto $H$ or $V$ before this transfer. 
Take for example permutation $ 3 2 4 6 1 7 9 8 5$. 
Trying to pop out the smallest element as soon as it is in the stacks leads to a dead-end. 
However, this permutation can be sorted using word 
$\rho_{3}\rho_{2}\lambda_{2}\rho_{4}\rho_{6}\rho_{1}\lambda_{1}\mu_{1}\mu_{2}\rho_{7}\lambda_{7}\lambda_{6}\lambda_{4}\lambda_{3}\mu_{3}\mu_{4}\rho_{9}\rho_{8}\rho_{5}\lambda_{5}\mu_{5}\mu_{6}\mu_{7}\lambda_{8}\mu_{8}\lambda_{9}\mu_{9}$.
But we prove that this natural idea of popping out smallest elements as soon as possible can be adapted considering right-to-left minima of the permutation.

We saw that a sorting process can be described as a word on the alphabet $\{\rlm\}$. 
In this article, we will also describe a sorting in a different way. 
Take the prefix of a sorting word, it corresponds to move some elements from the permutation to the stacks or output them. 
At the end of the prefix some elements may be in the stacks.
We can take a picture of the stacks and indeed, we will show that considering such pictures for all the prefixes 
that correspond to the entry of a right-to-left (RTL) minima of the permutation in $H$ is sufficient to decide the sortability. 
Such a picture is called a stack configuration.

\begin{defn}
A {\em stack configuration} $c$ is a pair of vectors $(v,w)$ of distinct integers 
such that the elements of $v$ (resp. of $w$) corresponds to the elements of $V$ (resp. of $H$) from bottom to top.

A stack configuration is {\em poppable} if elements in stacks $H$ and $V$ can be output in increasing order using operations $\lambda$ and $\mu$.
\end{defn}% TODO : 
% \marginpar{ajouter poppable partout où nécessaire}

Conditions for a stack configuration to be poppable have already been studied previously in \cite{Murphy02, PR13} and can be characterized by the following Lemma. Recall first that a permutation $\pi = \pi_{1}\pi_{2}\ldots\pi_{k}$ is a pattern of $\sigma = \sigma_{1}\sigma_{2}\ldots \sigma_{n}$ if and only if there exist indices $1 \leq i_{1} < i_{2} < \ldots < i_{k}$ such that $\sigma_{i_{1}}\sigma_{i_{2}}\sigma_{i_{3}}\ldots\sigma_{i_{k}}$ is order isomorphic to $\pi$.

\begin{lem}\label{lem:popable}
A stack configuration $c$ is poppable if and only if :
\begin{itemize}
\item Stack $H$ does not contain pattern $132$.
\item Stack $V$ does not contain pattern $12$.
\item Stacks $(V,H)$ does not contain pattern $|2|13|$.
\end{itemize}
Moreover, there is a unique way to pop the elements out in increasing order in terms of stack operations.
% TODO completer condition de pile
\end{lem}

The first two conditions are usual pattern relation, considering elements in the stack from bottom to top.
The third one means that there do not exist an element $i$ in $V$ and two elements $j,k$ in $H$ ($k$ above $j$) such that $j < i < k$.
There is a unique way to output those elements in increasing order as noticed in \cite{PR13}, 
so we will denote by $out_{c}(I)$ the word that consists in the operations necessary to output in increasing order elements of the set of values $I$ from a stack configuration $c$.

Notice that a stack configuration has no restriction upon its elements except that they must be different. 
Most of the time, a stack configuration will be associated to a permutation implying that the elements in the stacks are a subset of those of the permutation. 
In particular a {\em total stack configuration of $\sigma$} is a stack configuration in which the elements of the stacks are exactly  those of $\sigma$. 
% A configuration is called {\em accessible} if there exists a series of operations $\rho, \lambda, \mu$ leading to this configuration with $\sigma$ as input.

%In our algorithms, we push iteratively the elements of $\sigma$ into the stacks. 
%Sometimes, we will have to refer to the stack configuration as well as the next index $i$ of the element of $\sigma$ to be pushed. 
%We will call these pair $(c,i)$ where $c$ is a stack configuration and $i$ an integer, a stack configuration of $\sigma$. 
%Beware that the stack configuration $c$ could be not accessible by pushing the first elements $\sigma_{1} \ldots \sigma_{i-1}$ into the stacks. 
%Indeed, our major problem is to decide whether such configurations are accessible or not.

%Moreover, from a stack configuration $C$ many configurations are reachable applying rules $\lambda$ or $\mu$. 
%Indeed we define the closure of a stack configuration as :
%
%\begin{defn}
%The $\lambda$-closure (resp. $\lambda \mu$) of a stack configuration $C$ is the set of stacks configurations reachable from $C$ by applying $\lambda$ (resp. $\lambda \mu$).  
%Note that any configuration in the $\lambda$-closure has the same elements in the stacks. This is not the case for the $\lambda \mu$-closure.
%
%\end{defn}

In this article we often use decomposition of permutations into blocks. 
A block $B$ of a permutation $\sigma = \sigma_{1} \sigma_{2} \ldots \sigma_{n}$ is a factor $\sigma_{i} \sigma_{i+1} \ldots \sigma_{j}$ 
of $\sigma$ such that the set of values $\{ \sigma_{i}, \ldots, \sigma_{j}\}$ forms an interval. 
Notice that by definition of a factor, the set of indices $\{i, \dots, j\}$ also forms an interval.
Given two blocks $B$ and $B'$ of $\sigma$, we say that $B < B'$ if and only if $\sigma_{i} < \sigma_{j}$ for all $\sigma_i \in B$, $\sigma_j \in B'$. 
A permutation $\sigma$ is $\ominus$-decomposable if we can write it as $\sigma = B_{1} \ldots B_{k}$ 
such that $k \geq 2$ and for all $i$, $B_{i} > B_{i+1}$ in terms of blocks.
Otherwise we say that $\sigma$ is $\ominus$-indecomposable.
When each $B_i$ is $\ominus$-indecomposable, we write $\sigma = \ominus[B_{1},\ldots,B_{k}]$ and call it the $\ominus$-decomposition of $\sigma$. 
Notice that we do not renormalize elements of $B_i$ thus except $B_k$, the $B_i$ are not permutations.
Nevertheless, $B_i$ can be seen as a permutation by decreasing all its elements by $|B_{i+1}|+ \dots + |B_k|$.

The RTL (right-to-left) minima of a permutations are elements $\sigma_{k}$ such that there do not exist $j$ respecting $j > k$ and $\sigma_{j} < \sigma_{k}$.
We denote by $\sigma_{k_{i}}$ the $i^{\text{th}}$ right-to-left (RTL) minima of $\sigma$.
If $\sigma$ has $r$ RTL minima, then $\sigma = \dots \sigma_{k_{1}} \dots \sigma_{k_{2}} \dots \sigma_{k_{r}}$ with $\sigma_{k_{1}} = 1$ and $k_{r} = n$.

Take for example permutation $\sigma = 6\,5\,8\,7\,4\,1\,3\,2$. 
The $\ominus$-decomposition of $\sigma$ is $\sigma = \ominus[6\,5\,8\,7, 4, 1\,3\,2]$. 
Furthermore $\sigma$ has $2$ RTL-minima which are $\sigma_{6} = 1$ and $\sigma_{8} = 2$.

\begin{defn}
We denote $\sigma^{(i)} = \{ \sigma_{j} \mid j < k_{i} \text{ and } \sigma_{j} > \sigma_{k_{i}} \}$
the restriction of $\sigma$ to elements in the upper left quadrant of the $i^{\text{th}}$ right-to-left (RTL) minima $\sigma_{k_{i}}$.
The $\ominus_{i}$-decomposition of $\sigma$ is the $\ominus$-decomposition of $\sigma^{(i)} = \ominus[B_{1}^{(i)},\ldots, B_{s_{i}}^{(i)}]$. 
In the sequel $s_{i}$ always denote the number of blocks of $\sigma^{(i)}$ and $B_{j}^{(i)}$ the $j^{\text{th}}$ block in the $\ominus_{i}$-decomposition.
\end{defn}

There are two key ideas in this article. 
First, among all possible sorting words for a \sortable permutation, there always exists a sorting word respecting some condition denoted $\mathcal P$.
More precisely we prove that if a permutation $\sigma$ is sortable then there exists a sorting process in which 
the elements that lie in the stacks just before a right to left minima $k_{i}$ enters the stacks are exactly the elements of $\sigma^{(i)}$. 
A formal definition is given in Definition~\ref{def:propP}. 

The second idea is to encode the different sortings of a permutation respecting $\mathcal P$  by a sequence of graphs $\G{i}$ 
in which each node represents a stack configuration of a block $B_{j}^{(i)}$ and edges gives compatibility between the configurations.
The index $i$ is taken from $1$ to the number of right-to-left minima of the permutation. 

Section \ref{sec:general} study general properties on two-stack sorting and states which elements can move at each moment of a sorting process. 
Section \ref{sec:graph} introduces the sorting graph \G{i} which encode all the sortings of a permutation at a given time $t_{i}$ 
and gives an algorithm to compute this graph iteratively for all $i$ from 1 to the number of right-to-left minima. 
Last section focusses on complexity analysis.

\section{General results on two-stack sorting} \label{sec:general}
\subsection{Basic results}

We saw that a sorting process can be described as a word on the alphabet $\{\rlm\}$.
However not all words on the alphabet $\{\rlm\}$ describe sorting processes.

\begin{defn}[\stackw and \sortingw]
Let $\alpha \in \{\rlm\}$ and $w$ a word on the alphabet $\{\rlm\}$. 
Then $|w|_\alpha$ denotes the number of letters $\alpha$ in $w$.

A {\em \stackw} is a word $w \in \{\rlm\}^*$ such that for any prefix $v$ of $w$, $|v|_\rho \geq |v|_\lambda \geq |v|_\mu$.

A {\em \sortingw} is a \stackw $w$ such that $|w|_\rho = |w|_\lambda = |w|_\mu$.

For any permutation $\sigma$, a sorting word {\em for $\sigma$} is a \sortingw encoding a sorting process with $\sigma$ as input 
(leading to the identity of size $|\sigma|$ as output).
\end{defn}

Intuitively, \stackws are words describing some operations \rlm starting with empty stacks and an arbitrarily long input 
and they may be some elements in the stacks at the end of these operations, 
whereas \sortingws are words encoding a complete sorting process (stacks are empty at the beginning and at the end of the process). 

\begin{defn}[subword]
Let $I$ be a set of integers. 

For any decorated word $u$ we define $u_{|I}$ as the subword of $u$ made of letters $\rlmi$ with $i \in I$. 
For example, if $u= \rho_{3} \mu_{5} \lambda_{3} \rho_{6} \rho_{7} \lambda_{6}$ then $u_{|\{5, 6\}} = \mu_{5} \rho_{6} \lambda_{6}$.

We extend this definition to \stackws: given a permutation $\sigma$ and a \stackw $w$, 
$w_{|I}$ is the word of $\{\rlm\}^*$ obtaining from $\dec{w}{\sigma}_{|I}$ by deleting indices from letters $\rlmi$.

Intuitively, $w_{|I}$ is the subword of $w$ made of the operations of $w$ that act on integers of $I$
\end{defn}

\begin{lem}\label{lem:restrictStackword}
For any \stackw (resp. \sortingw) $w$, $w_{|I}$ is also a \stackw (resp. \sortingw).
\end{lem}

\begin{proof}
As $w$ is a \stackw, for all $i$ from $1$ to $|\sigma|$, 
$\rho_i$ appears before $\lambda_i$ which itself appears before $\mu_i$ in $\dec{w}{\sigma}_{|I}$.
Therefore for any prefix $v$ of $w_{|I}$, $|v|_\rho \geq |v|_\lambda \geq |v|_\mu$.
If moreover $w$ is a \sortingw, let $\alpha \in \{\rlm\}$, then for any letter $\alpha_i$ in $\dec{w}{\sigma}_{|I}$,
$\rho_i, \lambda_i$ and $\mu_i$ appear each exactly once in $\dec{w}{\sigma}_{|I}$ 
thus $|w_{|I}|_\rho = |w_{|I}|_\lambda = |w_{|I}|_\mu$.
\end{proof}

Now we turn to stack configurations, beginning with linking \stackws to \stackcs.

\begin{defn}[Action of a stack word on a permutation]
Let $w$ be a \stackw.
Starting with a permutation $\sigma$ as input, the stack configuration reached after performing operations described by the word $w$ 
is denoted $\conf_{\sigma}(w)$.
A stack configuration $c$ is {\em reachable} for $\sigma$ if there exists a \stackw $w$ such that $c = \conf_{\sigma}(w)$.
In other words a stack configuration is reachable for $\sigma$ if there exists a sequence of operations $\rho, \lambda, \mu$ 
leading to this configuration with $\sigma$ as input.
\end{defn}

\begin{lem}\label{lem:oneUponTheOther}
If $\sigma = \ominus[B_{1},\ldots B_{k}]$ then in any poppable stack configuration $c$ reachable for $\sigma$, 
elements of $B_{i}$ are below elements of $B_{j}$ in the stacks for all $i < j$ (see Figure \ref{fig:blockDecreasing}).
\end{lem}
\begin{proof}
Notice that by definition of a stack, elements of $H$ are in increasing order from bottom to top for the indices.
Moreover elements of $V$ are in decreasing order from bottom to top for their value
since from Lemma \ref{lem:popable} they avoid pattern $12$.
This leads to the claimed property.
\end{proof}

% The preceding lemma is illustrated by Figure \ref{fig:blockDecreasing} where blocks in decreasing positions are  stacked one upon the other.

\begin{figure}[ht]
\begin{center}
\begin{tikzpicture}[scale=.5]
\draw [very thick] (0,6) -- (0,0) -- (2,0) -- (2,6); 
\draw [very thick] (4,6) -- (4,0) -- (6,0) -- (6,6); 
\begin{scope}
\draw [thick] (0,1) -- (2,1);
\draw [thick] (0,2) -- (2,2);
\draw [thick] (0,4) -- (2,4);
\draw [thick] (0,5) -- (2,5);
\end{scope}
\begin{scope}[shift={(4,0)}]
\draw [thick] (0,1) -- (2,1);
\draw [thick] (0,2) -- (2,2);
\draw [thick] (0,4) -- (2,4);
\draw [thick] (0,5) -- (2,5);
\end{scope}
\begin{scope}[shift={(2,0)}]
\draw [dashed] (0,0) -- (2,0);
\draw [dashed] (0,1) -- (2,1);
\draw [dashed] (0,2) -- (2,2);
\draw [dashed] (0,4) -- (2,4);
\draw [dashed] (0,5) -- (2,5);
\end{scope}
\draw (1,0.5) node {$B_{1}$};
\draw (1,1.5) node {$B_{2}$};
\draw (1,4.5) node {$B_{k}$};
\draw (5,0.5) node {$B_{1}$};
\draw (5,1.5) node {$B_{2}$};
\draw (5,4.5) node {$B_{k}$};
\end{tikzpicture}
\caption{Poppable stack configuration reachable for $\ominus[B_{1},\ldots B_{k}]$.}\label{fig:blockDecreasing}
\end{center}
\end{figure}
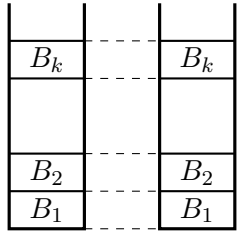

\begin{lem}\label{lem:petit}
Let $\sigma$ be a $2$-stack sortable permutation and $w = uv$ be a sorting word for $\sigma$. 
Assume that after performing operations of $u$, elements $1 \ldots i-1$ have been output and elements $i \ldots j$ are at the top of the stacks. 
Then there exists a sorting word $w' = u u' u''$ for $\sigma$ such that $u'$ consists only in moving elements $i \ldots j$ 
from the stacks to the output in increasing order without moving any other elements.

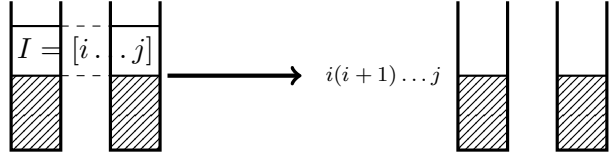
\begin{center}
\begin{tikzpicture}[scale=.33]
\begin{scope}
\draw [very thick] (0,6) -- (0,0) -- (2,0) -- (2,6); 
\draw [very thick] (4,6) -- (4,0) -- (6,0) -- (6,6); 
\begin{scope}
\draw [thick] (0,3) -- (2,3);
\draw [thick] (0,5) -- (2,5);
\end{scope}
\begin{scope}[shift={(4,0)}]
\draw [thick] (0,3) -- (2,3);
\draw [thick] (0,5) -- (2,5);
\end{scope}
\begin{scope}[shift={(2,0)}]
\draw [dashed] (0,3) -- (2,3);
\draw [dashed] (0,5) -- (2,5);
\end{scope}
\draw (3,4) node {$I=[i\ldots j]$};
\draw[pattern=north east lines] (0,0)  rectangle (2,3);
\draw[pattern=north east lines] (4,0)  rectangle (6,3);
\end{scope}
\draw (15,3) node {\scriptsize $i (i+1) \ldots j$};
\begin{scope}[shift={(18,0)}]
\draw [very thick] (0,6) -- (0,0) -- (2,0) -- (2,6); 
\draw [very thick] (4,6) -- (4,0) -- (6,0) -- (6,6); 
\begin{scope}
\draw [thick] (0,3) -- (2,3);
\end{scope}
\begin{scope}[shift={(4,0)}]
\draw [thick] (0,3) -- (2,3);
\end{scope}
\draw[pattern=north east lines] (0,0)  rectangle (2,3);
\draw[pattern=north east lines] (4,0)  rectangle (6,3);
\end{scope}
\draw[ultra thick,->] (6.3,3) -- (11.7,3);
\end{tikzpicture}
\end{center}
\end{lem}

\begin{proof}
We claim that $u' = v_{|[i \ldots j]}$ and $u'' = v_{|![i \ldots j]}$ satisfy the desired property, 
where $![i \ldots j]$ is the set of integers $[1 \ldots |\sigma|] \setminus [i\ldots j]$.
This can be checked using decorated words associated to $w$ and $w'$ and noticing that 
$v_{|[i \ldots j]} = out_{c_{\sigma}(u)}([i\ldots j])$ and $v_{|![i \ldots j]} = v_{|>j}$
since by hypothesis after performing operations of $u$, 
elements $1 \ldots i-1$ have been output and elements $i \ldots j$ are at the top of the stacks.
\end{proof}

The stack configurations for a sorting process encode the elements that are currently in the stacks. 
But some elements are still waiting in the input and some elements have been output.
To fully characterize a configuration, we define an {\em extended} stack configuration of a permutation $\sigma$ of size $n$
to be a pair $(c,i)$ where $i \in \{1, \dots n+1\}$ and $c$ is a poppable stack configuration 
made of all elements within $\sigma_{1}, \sigma_{2},\ldots, \sigma_{i-1}$ that are greater than a value $p$. 
The elements $\sigma_{i},\ldots, \sigma_{n}$ are waiting to be pushed and elements $\sigma_{j} < p, j < i$ have already been output.
Notice that we don't need the configuration to be reachable.
% TODO : 
% \marginpar{vérifier que c'est bien le cas}

\begin{defn}
Let $\sigma$ be a permutation and $(c,i)$ be an extended stack configuration of $\sigma$.
Then the extended stack configuration $(c',j)$ of $\sigma$ is {\em accessible} from $(c,i)$  if 
the stack configuration $(c',j)$ can be reached starting from $(c,i)$ and performing operations $\rho,\lambda$ and $\mu$
such that moves $\mu$ perfomed output elements of $c \cup \{\sigma_i \dots \sigma_n\}$ in increasing order.
\end{defn}

For example, if $\sigma = 2\,3\,1\,6\,5\,8\,4\,7$ then $(\tikz[baseline,scale=.3]{\ssConf{6,5}{8}},7)$ is accessible from 
$(\tikz[baseline,scale=.3]{\ssConf{}{3,2}},4)$ by the sequence of operations $\mu_{2}\mu_{3}\rho_{6}\rho_{5}\rho_{8}\lambda_{8}$. 
But $(\tikz[baseline,scale=.3]{\ssConf{6}{3,2}},5)$ is not accessible from $(\tikz[baseline,scale=.3]{\ssConf{3}{2,1}},4)$.

Indeed notice that the question of whether a permutation is \sortable can be reformulated as : \\
{\bf Is $(\tikz[baseline,scale=.3]{\ssConf{}{}}, n+1)$ accessible from $(\tikz[baseline,scale=.3]{\ssConf{}{}}, 1)$ ?}
 
To solve this problem is the main goal of this article and is somehow hard, however some special cases are easier to deal with. 
The following Lemma give conditions on the involved configurations under which the compatibility decision problem is linear
and can by solved by the {\em isAccessible} procedure given in Algorithm~\ref{algo:compatibleRectangle}. 
In the last sections, we show how more general cases can be solved using this Lemma.

\begin{lem}\label{lem:compatibleRectangle}
Let $\sigma$ be a permutation of size $n$ and $(c,i)$, $(c',j)$ two extended stack configurations of $\sigma$ with $i < j$. 
Let $E$ be the set of elements of $c$ and $F$ those of $c'$.
%and $E,F$ sets of elements of $\sigma$. 
\begin{itemize}
\item If there  exists $k,\ell \in \{1 \dots n\}$ such that $E = \{\sigma_m \mid m \leq k\}$ and $F = \{\sigma_m \mid \sigma_m \geq \ell \}$
\item If moreover $E \cup F = \sigma$
\end{itemize}
Then we can decide in linear time whether $(c',j)$ is accessible from $(c,i)$ using Algorithm \ref{algo:compatibleRectangle}.
\end{lem}

\begin{proof}
We  prove by case study that there is no choice between operations $\rho, \lambda, \mu$ at each time step. 
This is illustrated by Algorithm \ref{algo:compatibleRectangle}. 
We first prove its correctness before studying its complexity.

We start with configuration $curr = c$. 
By studying specific elements of the current configuration $curr$, 
we prove that we can always decide which operation should be performed to transform $curr$ into $c'$.  
If at any step this operation is forbidden then $c'$ is not accessible from $curr$.
Thus repeating the following process will eventually lead to decide whether $c'$ is accessible from $c$.

Notice that by definition, $c$ and $c'$ are poppable thus $curr$ has to be poppable, hence to avoid the three patterns of Lemma~\ref{lem:popable}.
Let $p$ be the next element to be output, {\em i.e.} the smallest element of $c \cup \{\sigma_i \dots \sigma_n\}$. 
Let $\sigma_{H}$ (resp.~$\sigma_V$) be the topmost element of $H$ (resp.~of $V$) and $\sigma_{q}$ be the element waiting in the input to be pushed onto $H$
($\sigma_{q}$ may not exist and in that case $\sigma_{q} = \varnothing$;
at the beginning $\sigma_{q} = \sigma_i$).

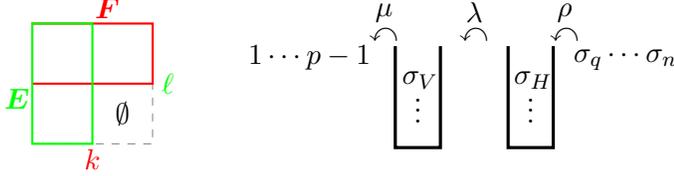
\begin{tikzpicture}[scale=.4]
\useasboundingbox (-3,-0.5) rectangle (8,4.5);
\draw[help lines, dashed] (0,0) rectangle (4,4);
\draw[thick, red] (0,2) rectangle (4,4);
\draw[red] (2.5,4.5) node {{\boldmath$F$}};
\draw[red] (2,-0.5) node {$k$};
\draw[thick, vert] (0,0) rectangle (2,4);
\draw[vert] (-0.5,1.5) node {{\boldmath$E$}};
\draw[vert] (4.5,2) node {$\ell$};
\draw (3,1) node {$\emptyset$};
\end{tikzpicture}
\begin{tikzpicture}[scale=.3]
\useasboundingbox (-5,-1.5) rectangle (11,5.5);
\draw[very thick] (5,3.5) -- (5,-1) -- (7,-1) -- (7,3.5);
\draw[very thick] (0,3.5) -- (0,-1) -- (2,-1) -- (2,3.5);
\draw (10.2,2.9) node {$\sigma_q \cdots \sigma_n$};
\draw (7.5,4) node {$\curvearrowleft$};
\draw (7.5,5) node {$\rho$};
\draw (3.5,4) node {$\curvearrowleft$};
\draw (3.5,5) node {$\lambda$};
\draw (-0.5,4) node {$\curvearrowleft$};
\draw (-0.5,5) node {$\mu$};
\draw (-3.8,3) node {$1 \cdots p-1$};
\draw (1.1,2) node {$\sigma_V$};
\draw (1,1) node {$\vdots$};
\draw (6.1,2) node {$\sigma_H$};
\draw (6,1) node {$\vdots$};
\end{tikzpicture}

\begin{itemize}
\item If $\sigma_{V} = p$ then we perform $\mu$ thanks to Lemma \ref{lem:petit}.
\item Otherwise operation $\mu$ is forbidden. We have to chose between $\rho$ and $\lambda$.
Moreover $p \notin V$ as $V$ is in decreasing order from bottom to top.
\begin{enumerate}
\item Suppose that $\sigma_{H} < \ell$. 
This means that $\sigma_{H} \not\in F$ {\em i.e.} $\sigma_{H} \not\in c'$. 
Notice that by definition of $p$, $p \leq \sigma_{H}$ thus $p \not\in c'$.
Moreover $p \notin V$ thus $p \in H$. 
If $p = \sigma_{H}$ then, by Lemma \ref{lem:petit}, we can pop out $p$. 
Thus we perform $\lambda$. 
If $\sigma_{H} \not= p$, then we will prove that all elements $x$ such that $p \leq x \leq \sigma_{H}$ form an interval at the top of the stacks. 
Those elements are all in the stacks by definition of $\ell$ and $p$. 
As $V$ is decreasing, the elements of $[p\ldots \sigma_{H}]$ belonging to $V$ are at the top of it. 
Consider now the position of those elements in $H$. 

Suppose that it is not an interval. 
Then it exists an element $x$ in $H$ such that $x < \sigma_{H}$ and there is an element $y > \sigma_{H}$ between $x$ and $\sigma_H$. 
But in that case, elements $x y \sigma_{H}$ form the pattern $1\,3\,2$ and $curr$ is not poppable 
so any movement is allowed here $\rho, \lambda$ or $\mu$ because we will never reach $c'$.

Suppose now that the elements $[p\ldots \sigma_{H}]$ form an interval in $H$ and $V$. 
Then as $p \in H$ is the smallest element, by Lemma \ref{lem:petit}, we want to pop out elements $[p\ldots \sigma_{H}]$, hence we perform $\lambda$.

In conclusion, if $\sigma_{H} < \ell$ we perform $\lambda$.

\item If not, then  $\sigma_{H} \geq \ell$ and thus $\sigma_{H} \in c'$. Once again there are different cases:
\begin{enumerate}
	\item If $H = \varnothing$ then $\lambda$ is forbidden, thus we perform $\rho$.
	\item If $\sigma_{H} \in H(c')$, it must stay in $H$ thus $\lambda$ is forbidden and we perform $\rho$.
	\item Else $\sigma_{H} \in V(c')$. 
	\begin{itemize}
		\item If $\sigma_{q} \in H(c')$ then $\rho$ is forbidden because $\sigma_{q}$ would prevent $\sigma_{H}$ from moving. 
                      Thus we perform $\lambda$.
		\item Else $\sigma_{q} \in V(c')$. 
                      If $\sigma_{H} > \sigma_{q}$, as $\sigma_{H} \in V(c')$, $\rho$ is forbidden otherwise we cannot put $\sigma_{q}$ above $\sigma_{H}$ in $V$. 
                      Thus we perform $\lambda$.
		\item Otherwise $\sigma_{H}, \sigma_{q} \in V(c')$ and $\sigma_{H} < \sigma_{q}$. 
                      $\lambda$ is forbidden otherwise we cannot put $\sigma_{H}$ above $\sigma_{q}$ in $V$. Thus we perform $\rho$.
	\end{itemize}
\end{enumerate}
\end{enumerate}
\end{itemize}
We have proved that at each step of the algorithm, we know which move we have to do if we want to reach $c'$.
Moreover while $q < j$ or $p < \ell$ or $\sigma_{H} \in V(c')$, it is impossible that $curr = c'$ so we have to continue.
Conversely if $q \geq j$ and $p \geq \ell$ and $\sigma_{H} \notin V(c')$ then $\rho$ and $\mu$ and $\lambda$ are forbidden and we have to stop.
Then if $curr = c'$, $c'$ is accessible from $c$, otherwise $c'$ is not accessible from $c$.

Finally there are at most $3n$ steps since at each step of the algorithm we perfom a move $\rho$, $\lambda$ or $\mu$.
Moreover each step takes a constant time, therefore the algorithm runs in linear time.
\end{proof}

%\marginpar{Pb indiceD dans l algo}
\begin{algorithm}[H]
% \DontPrintSemicolon
\KwData{$\sigma$ a permutation and $(c,i),(c',j)$ two stack configurations of $\sigma$ respecting conditions of Lemma~\ref{lem:compatibleRectangle}}
\KwResult{{\tt true} or {\tt false} depending whether the configuration $c'$ is accessible from $c$} 
\Begin{
	Put configuration $c$ in the stacks $H$ and $V$\;
	$p \leftarrow $ the smallest element of $c \cup \{\sigma_i \dots \sigma_n\}$ (next element to be output)\;
	$q \leftarrow i$ (next index of $\sigma$ that must enter the stacks)\;
	\While{$q < j$ OR $p < \ell$ OR $\sigma_{H} \in V(c')$}{
		\eIf{$\sigma_{V} = p$}{
			Perform $\mu$; $p \leftarrow p+1$\;
		}{
			\eIf{$\sigma_{H} < \ell$}{
				Perform $\lambda$\;
			}{
				\eIf{$H = \varnothing$ OR $\sigma_{H} \in H(c')$}{
					Perform $\rho$; $q \leftarrow q+1$\;
				}{
					\eIf{$\sigma_{q} \in H(c')$ OR $\sigma_{H} > \sigma_{q}$}{
						Perform $\lambda$\;
					}{
						Perform $\rho$; $q \leftarrow q+1$\;
					}
				}
			}
		}
	}
	Return $(H,V) == c'$\;
}
\caption{isAccessible$\big( (c,i),(c',j), \sigma \big)$}\label{algo:compatibleRectangle}
\end{algorithm}

%\begin{rem}
%The hypothesis of Lemma \ref{lem:compatibleRectangle} deal with a configuration $\sigma$. 
%Note that the previous lemma could also be applied on sub permutations (non-contiguous elements) up to a renormalization. 
%Notice also that the stacks could also contain elements at the bottom of them. 
%This indeed be seen as those elements are in common between $c$ and $c'$ and are in the same stack between those two configurations. 
%\end{rem}

In the sequel of this article, we do not compute all possible stack configurations during a sorting process of a given permutation $\sigma$
but indeed focus on specific steps of the sorting. 
We study the possible stack configurations at each time step $t_{i}$
corresponding to the moment just before the right to left minimum $\sigma_{k_{i}}$ is pushed onto stack $H$.
Those configurations are configurations $(c,k_{i})$ accessible from $(\tikz[baseline,scale=.3]{\ssConf{}{}}, 1)$. 

We will prove that we can add two different restrictions on these configurations.
First, $(c,k_{i})$ must be a pushall stack configuration of $\sigma^{(i)}$ (see below).
Second $(c,k_{i})$ must be an evolution of some configuration $(c',k_{i-1})$ between time $t_{i-1}$ and $t_{i}$.

%and we keep track of the evolution of those configurations through time to decide whether 
%$(c,k_{i})$ is compatible with $(\tikz[baseline,scale=.3]{\ssConf{}{}}, n+1)$.
%and  $(c,k_{i})$ is compatible with $(\tikz[baseline,scale=.3]{\ssConf{}{}}, n+1)$.
%Indeed, this corresponds also to configuration $\conf_{\sigma}(u)$ such that there exists $v$ such that $u\rho_{k_{i}}v$ is a sorting word for $\sigma$.

\begin{defn}[pushall configuration]
A stack configuration is a {\em pushall} stack configuration of $\sigma$ if it is poppable, total and reachable for $\sigma$.
\end{defn}

\subsection{From time $t_{i}$ to time $t_{i+1}$}

Thanks to the previous decomposition into different time steps corresponding to each moment a right-to-left minima is pushed onto $H$ 
and our previous work \cite{PR13} on \pushall permutations,
we can give a polynomial algorithm deciding whether a permutation is \sortable.
Indeed, we will prove that it is enough to consider configurations such that for each $t_{i}$ 
the only elements in the stacks are exactly those of $\sigma^{(i)}$.
But $\sigma^{(i)}$ is a permutation that ends with its smallest element such that a sorting consists in 
pushing all elements into the stacks then popping all elements out.
Those possibilities are described in \cite{PR13} where Proposition 4.8 gives all possible pushall stack configurations.
When a permutation is $\ominus$-indecomposable, 
Theorem 4.4 of \cite{PR13} states that the number of possible pushall stack configurations is linear in the size of the permutation.
This will ensure that our algorithm runs in polynomial time.
Using this result, we now have the possible total stack configurations at time $t_{1}$.

The key idea for computing possible stack configurations at time $t_i$ relies on Lemma~\ref{lem:compacc}. 
Informally, it is possible to decide whether a configuration at time $t_{i}$ can evolved into a specific configuration at time $t_{i+1}$. 
Moreover, during this transition, only a few moves are undetermined. 
Indeed the largest elements won't move, the smallest one will be pushed accordingly to \cite{PR13} 
and the remaining ones form a $\ominus$-indecomposable permutation that will allow us to exhibit a polynomial algorithm.

First of all we denote by $A^{(i)}$ the common part of the permutations $\sigma^{(i)}$ and $\sigma^{(i+1)}$, 
that is, $A^{(i)} = \sigma^{(i)} \bigcap \sigma^{(i+1)} = \{ \sigma_{j} \mid j < k_{i} \text{ and } \sigma_{j} > \sigma_{k_{i+1}} \}$. 
This subpermutation $A^{(i)}$ intersects $\ominus$-indecomposable blocks of $\sigma^{(i)}$ and $\sigma^{(i+1)}$. 
Let $p^{(i)}$ (resp. $q^{(i+1)}$) be the index such that $B_{p^{(i)}}^{(i)}$ (resp. $B_{q^{(i+1)}}^{(i+1)}$) contains the smallest value of $A^{(i)}$. 
Let $D^{(i)} =  (B_{p^{(i)}}^{(i)} \bigcup B_{q^{(i+1)}}^{(i+1)}) \bigcap A^{(i)}$.

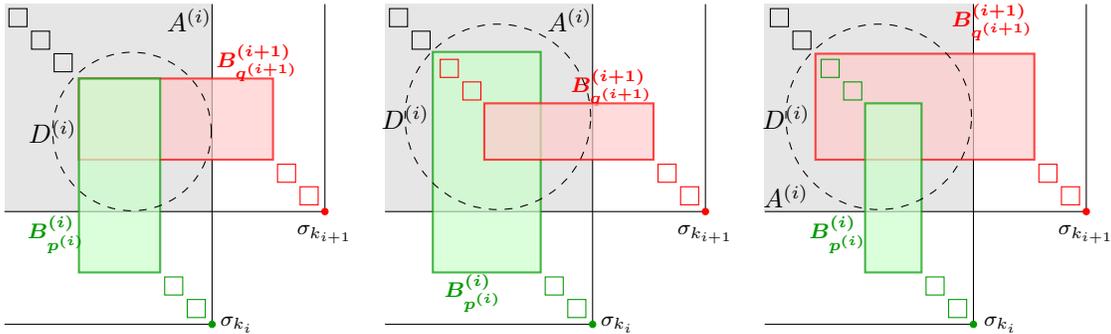
\begin{tikzpicture}[scale=.3]
\draw[draw=none, fill=black!10!white] (-1.2,5) rectangle (8,14.2);
\draw (7,13.5) node{{\small $A^{(i)}$}};
%\useasboundingbox (0,-1.5) rectangle (4,4);
\draw (-1.2,0) -- (8,0);
\draw (-1.2,5) -- (13,5);
\draw (8,0) -- (8,14.2);
\draw (13,5) -- (13,14.2);
\draw [green!60!black] (8,0) [fill] circle (4pt);
\draw (9.1,0) node {{\scriptsize $\sigma_{k_i}$}};
\draw [red] (13,5) [fill] circle (4pt);
\draw (13,4) node {{\scriptsize $\sigma_{k_{i+1}}$}};
\draw [green!60!black] (6.9,1.1) rectangle +(0.8,-0.8);
\draw [green!60!black] (5.9,2.1) rectangle +(0.8,-0.8);
\draw [red] (11.9,6.1) rectangle +(0.8,-0.8);
\draw [red] (10.9,7.1) rectangle +(0.8,-0.8);
%\draw (4.5,8.5) rectangle +(0.8,-0.8);
%\draw (3.5,9.5) rectangle +(0.8,-0.8);
%\draw (2.5,10.5) rectangle +(0.8,-0.8);
\draw (1,12) rectangle +(0.8,-0.8);
\draw (0,13) rectangle +(0.8,-0.8);
\draw (-1,14) rectangle +(0.8,-0.8);
\draw[thick, red,fill=red!20,opacity=.7] (10.7,7.3) rectangle +(-8.6,3.6);
\draw[red] (10,11.6) node {{\scriptsize \boldmath$B^{(i+1)}_{q^{(i+1)}}$}};
\draw[thick, green!60!black,fill=green!20,opacity=.7] (5.7,2.3) rectangle +(-3.6,8.6);
\draw[green!60!black] (1.1,4) node {{\scriptsize \boldmath$B^{(i)}_{p^{(i)}}$}};
\draw [dashed]  (4.45,8.55) circle  (100pt) node[left=17pt] {{\small$ D^{(i)}$}};
\end{tikzpicture}
\begin{tikzpicture}[scale=.3]
%\useasboundingbox (0,-1.5) rectangle (4,4);
\draw[draw=none, fill=black!10!white] (-1.2,5) rectangle (8,14.2);
\draw (7,13.5) node{{\small$A^{(i)}$}};
\draw (-1.2,0) -- (8,0);
\draw  (-1.2,5) -- (13,5);
\draw (8,0) -- (8,14.2);
\draw  (13,5) -- (13,14.2);
\draw[green!60!black ] (8,0) [fill] circle (4pt);
\draw  (9.1,0) node {{\scriptsize $\sigma_{k_i}$}};
\draw [red]  (13,5) [fill] circle (4pt);
\draw  (13,4) node {{\scriptsize $\sigma_{k_{i+1}}$}};
\draw[red] (8.85,10.5) node {{\scriptsize \boldmath$B^{(i+1)}_{q^{(i+1)}}$}};
\draw[thick, green!60!black, fill=green!20,opacity=.7] (5.7,2.3) rectangle
+(-4.78,9.78);
\draw[thick, red, fill=red!20, opacity=.7] (10.7,7.3) rectangle +(-7.5,2.5);
\draw[green!60!black] (2.7,1.4) node {{\scriptsize \boldmath$B^{(i)}_{p^{(i)}}$}};
\draw [green!60!black] (6.9,1.1) rectangle +(0.8,-0.8);
\draw [green!60!black] (5.9,2.1) rectangle +(0.8,-0.8);
\draw  [red]  (11.9,6.1) rectangle +(0.8,-0.8);
\draw  [red] (10.9,7.1) rectangle +(0.8,-0.8);
%\draw (4.5,8.5) rectangle +(0.8,-0.8);
%\draw (3.5,9.5) rectangle +(0.8,-0.8);
\draw [red] (2.25,10.75) rectangle +(0.8,-0.8);
\draw [red] (1.25,11.75) rectangle +(0.8,-0.8);
\draw (0,13) rectangle +(0.8,-0.8);
\draw (-1,14) rectangle +(0.8,-0.8);	
\draw [dashed]  (3.8,9.2) circle  (117pt) node[left=22pt] {{\small$D^{(i)}$}};
\end{tikzpicture}
\begin{tikzpicture}[scale=.3]
%\useasboundingbox (0,-1.5) rectangle (4,4);
\draw[draw=none, fill=black!10!white] (-1.3,5) rectangle (8,14.2);
\draw (-.3,5.7) node{{\small$A^{(i)}$}};
\draw (-1.25,0) -- (8,0);
\draw (-1.25,5) -- (13,5);
\draw (8,0) -- (8,14.2);
\draw (13,5) -- (13,14.2);
\draw [green!60!black](8,0) [fill] circle (4pt);
\draw (9.2,0) node {{\scriptsize $\sigma_{k_i}$}};
\draw [red] (13,5) [fill] circle (4pt);
\draw (13,4) node {{\scriptsize $\sigma_{k_{i+1}}$}};
\draw [green!60!black](6.9,1.1) rectangle +(0.8,-0.8);
\draw [green!60!black] (5.9,2.1) rectangle +(0.8,-0.8);
\draw [red] (11.9,6.1) rectangle +(0.8,-0.8);
\draw [red] (10.9,7.1) rectangle +(0.8,-0.8);
%\draw (4.5,8.5) rectangle +(0.8,-0.8);
%\draw (3.5,9.5) rectangle +(0.8,-0.8);
\draw (0,13) rectangle +(0.8,-0.8);
\draw (-1,14) rectangle +(0.8,-0.8);
%\draw (-1.25,14.25) rectangle +(0.8,-0.8);
%\draw (-2.25,15.25) rectangle +(0.8,-0.8);
%\draw (-3.5,16.55) rectangle +(0.8,-0.8);
\draw[thick, red, fill=red!20, opacity=.7]  (10.7,7.3) rectangle +(-9.7,4.7);
\draw[red] (8.85,13.4) node {{\scriptsize \boldmath$B^{(i+1)}_{q^{(i+1)}}$}};
\draw[thick, green!60!black, fill=green!20,opacity=.7](5.7,2.3) rectangle +(-2.5,7.5);
\draw[green!60!black] (2,4) node {{\scriptsize \boldmath$B^{(i)}_{p^{(i)}}$}};
\draw [green!60!black](2.25,10.75) rectangle +(0.8,-0.8);
\draw [green!60!black](1.25,11.75) rectangle +(0.8,-0.8);
\draw [dashed]  (3.8,9.2) circle  (117pt) node[left=22pt] {{\small$D^{(i)}$}};
\end{tikzpicture}

\begin{lem}\label{lem:blocks}
For any $j < \min (p^{(i)}, q^{(i+1)})$, $B_{j}^{(i)} = B_{j}^{(i+1)}$.
\end{lem}

\begin{lem}\label{lem:HautGaucheA}
Let $\sigma_{\ell} \in A^{(i)}$. During a sorting process of $\sigma$, elements $\sigma_{m}$ such that $\sigma_{m} > \sigma_{\ell}$ and $m < \ell$  do not move between $t_{i}$ and $t_{i+1}$.
\end{lem}

\begin{proof}
Let $\sigma_{m}$ be an element such that $m< \ell$ and $\sigma_{m} > \sigma_{\ell}$. 
As $\sigma_{\ell} \in A^{(i)}$, $\sigma_{\ell} > \sigma_{k_{i+1}}$ and $j < k_{i}$, so does $\sigma_{m} > \sigma_{k_{i+1}}$ and $m < k_{i}$. 
Hence both elements $\sigma_{m}, \sigma_{\ell}$ lie in the stacks between $t_{i}$ and $t_{i+1}$ 
(they cannot be output as $\sigma_{k_{i+1}}$ must be output first).
Suppose that $\sigma_{m}$ is in $H$ at time $t_{i}$. 
As $m < \ell$, element $\sigma_{\ell}$ is pushed after $\sigma_{m}$ into the stacks, 
thus either $\sigma_{\ell}$ is above $\sigma_{m}$ in $H$ or lies in $V$ at time $t_{i}$ and $t_{i+1}$.
So, $\sigma_{m}$ cannot move into $V$, otherwise $\sigma_{\ell}$ would be under it in $V$ and $V$ would contain a pattern $12$.
So, $\sigma_{m}$ stay in $H$.

Suppose now that $\sigma_{m}$ is in $V$ at time $t_{i}$.
As noticed previously, this element is not output at time $t_{i+1}$.
So it also lies in stack $V$ at time $t_{i+1}$, proving the lemma.
\end{proof}

% Using Lemma \ref{lem:HautGaucheA}, Lemma \ref{lem:compatibleRectangle} can be generalized to deal with more complex situations.
% 
% \begin{lem}\label{lem:compatible}
% Let $\sigma$ be a permutation of size $n$ whose diagram can be decomposed into $4$ blocks 
% \begin{tikzpicture}[scale=.4]
% \useasboundingbox (1,1) rectangle (3,3);
% \draw (1,1) grid +(2,2);
% \draw (1.5,1.5) node {$E$};
% \draw (1.5,2.5) node {$D$};
% \draw (2.5,2.5) node {$F$};
% \draw (2.5,1.5) node {$\varnothing$};
% \end{tikzpicture},
% $j = \max \{i \mid \sigma_i \in D \cup E \}$
% and $B$ be the last block of the $\ominus$-decomposition of $\sigma_{|D}$.
% Let $c$ (resp. $c'$) is a stack configurations of $\sigma$ containing all elements of $D \cup E$ (resp. $D \cup F$).
% If $(c',n+1)$ is accessible from $(c,j+1)$, then $c'_{D\setminus B} = c_{D\setminus B}$.
% \end{lem}
% \begin{proof}
% If $D\setminus B$ is empty, then the result trivially holds.
% Otherwise, let $\sigma_{m} \in D\setminus B$ and $\sigma_{\ell} \in B$.
% Then by definition of $B$, $\sigma_{m} > \sigma_{\ell}$ and $m < \ell$
% If $(c',n+1)$ is accessible from $(c,j+1)$, consider in evolution from $(c,j+1)$ to $(c',n+1)$.
% We claim that $\sigma_{m}$ 
% % TODO : N'est pas vrai non plus, a adapter en tenant compte des RTL minima.
% % En fait on n'a pas besoin de ce lemme, on peut s'en sortir en utilisant juste le suivant et le précédent.
% % TODO : Virer ce lemmes et adapter les références qu'on y fait.
% \end{proof}

%Le but est de partir d un tri de $\sigma_{<k_{i}}$ et de le compléter pour arriver a un tri de $\sigma_{<k_{i+1}}$.

In the following we study conditions for $2$ total pushall stack configurations $c$ and $c'$ 
corresponding to stack configuration of $\sigma^{(i)}$ and $\sigma^{(i+1)}$ to be accessible one from the other, 
that is, if we can move elements starting from $c$ at time $t_{i}$ to obtain $c'$ at time $t_{i+1}$.

\begin{lem}\label{lem:compacc}
Let $(c,k_{i})$ (resp.~$(c',k_{i+1})$) be a total stack configuration of $\sigma^{(i)}$ (resp.~$\sigma^{(i+1)}$). 
Let $\pi = \sigma_{|B^{(i)}_{p^{(i)}} \bigcup B^{(i+1)}_{q^{(i+1)}}}$ then 
$(c',k_{i+1})$ is accessible from $(c,k_{i})$ for $\sigma$ iff:
\begin{enumerate}
\item $(c'_{|\pi}, |\pi|+1)$ is accessible from $(c_{|\pi}, \sharp (D^{(i)} \bigcup B_{p^{(i)}}^{(i)}) +1)$ for $\pi$.
\item $\forall j < min(p^{(i)},q^{(i+1)})$, $c_{|B_{j}^{(i)}} = c'_{|B_{j}^{(i)}}$.
\item $\forall j > q^{(i+1)}, c'_{|B_{j}^{(i+1)}}$ is a reachable configuration.
\end{enumerate}
\end{lem}

\begin{proof}
Suppose first that $(c',k_{i+1})$ is accessible from $(c,k_{i})$. 
This means that we can go from $c$ to $c'$ using operations represented by the decorated word $\hat{w}$. 
These operations are stable that is for all $I$, $c'_{|I}$ is accessible from $c_{|I}$.
To do so, we just extract operations corresponding to elements of $I$. Indeed the decorated word $\hat{w}_{|I}$ allow to transform $c$ into $c'$. 
This proves the first point of Lemma \ref{lem:compacc}.

Let $\sigma_{\ell} \in B_{p}^{(i)}$.
Lemma \ref{lem:HautGaucheA} ensures that elements of $B^{(i)}_{j}$ with $j < p^{(i)}$
do not move between $t_{i}$ and $t_{i+1}$ proving the second point of Lemma \ref{lem:compacc}.

Finally, elements of $B_{j}^{(i+1)}$ for $j > q^{(i+1)}$ are pushed iteratively when going from $c$ to $c'$. 
Those elements stay in the stacks as $\sigma_{k_{i+1}}$ which is smaller is pushed after them. 
Thus they correspond to a pushall configuration.

Conversely, suppose that we have the $3$ different points above, we must prove that $(c',k_{i+1})$ is accessible from $(c,k_{i})$ for $\sigma$.
We start by taking the stack configuration $c$ and we will prove that we can obtain $c'$ by moving elements. 
First of all, as $c$ is a pushall stack configuration, 
and as elements of $B_{\ell}$ for $\ell > p$ are the smallest one and have been pushed last into the stacks they are at the top of the stacks 
(see Lemma~\ref{lem:oneUponTheOther}). 
Thus we can pop them and output them in increasing order using Lemma~\ref{lem:petit}. 

The remaining elements in the stacks don't move in the preceding operation, thus stay in the same position than in $c$. 
In that configuration, elements of $B_{p^{(i)}}^{(i)}$ are the smallest ones and have been pushed the latter in the stacks. 
Hence they lie at the top of the stacks. 

Then using point $1$ of our hypothesis, we can move those elements together with pushing elements of $B^{(i+1)}_{q^{(i+1)}} \setminus B^{(i)}_{p^{(i)}}$ 
so that all those elements (that is elements of $\pi$) are in the same position than in $c'$. 
Then, by hypothesis item $3$, $\forall j > q^{(i+1)}, c'_{|B_{j}^{(i+1)}}$ is a reachable configuration. 
Thus we can push its elements into the stacks in the same relative order than in $c'$ (see Lemma~\ref{lem:oneUponTheOther}). 
During these operations we ensure that elements of $B_{\ell}$ with $\ell \geq min(p^{(i)},q^{(i+1)})$, $c_{|B_{j}^{(i)}}$ 
are in the same position in our configuration than in $c'$. 
Point $2$ ensures that we indeed obtain $c'$.
\end{proof}

The preceding Lemma describes exactly which elements can move between $t_{i}$ and $t_{i+1}$ and how they move. 
But the hypothesis of Lemma \ref{lem:compacc} are restrictive 
that is configurations $c$ and $c'$ must be two total stack configurations of $\sigma^{(i)}$ and $\sigma^{(i+1)}$. 
Thus, we first prove that among all sortings of a 2-stack sortable permutation, 
there exists at least one for which the stack configurations at time $t_{i}$ contains exactly the elements of $\sigma^{(i)}$ for all $i$. 

\begin{defn}[Properties $(P_{i})$ and $(P)$]\label{def:propP}
Let $\sigma$ be a permutation and $w$ a sorting word for $\sigma$. $w$ verifies $(P_{i})$ if and only if 
\begin{enumerate}[(i)]
\item  $\rho_{\sigma_{k_{i}}} \lambda_{\sigma_{k_{i}}} \mu_{\sigma_{k_{i}}}$ is a factor of $w$.
\item $\mu_{\sigma_{j}}$ appears before $\rho_{\sigma_{k_{i}}}$ for all $\sigma_{j} < \sigma_{k_{i}}$.
\item All operations $\mu_{\sigma_{\ell}}$ with $\sigma_{\ell} \in B_{j}^{(i)}$ and $j \in [p^{(i)}+1 .. s_{i}]$
appear before $\rho_{\sigma_{k_{i}+1}}$ in $w$.
% = \min\{ k_{i}, \max \{ \ell, \exists m , \sigma_{{k_{i}+1} } \in B^{(i)}_{m} \text{ and } \sigma_{\ell} \in B^{(i)}_{m} \} \}$ 
\end{enumerate}
where $\sigma_{k_{i}}$ is the $i^{th}$ right to left minima of the permutation and $\sigma^{(i)} = \ominus [ B^{(i)}_{1}, \ldots, B^{(i)}_{s_{i}} ]$.

If a word $w$ verifies Property $(P_{i})$ for all $i$ then we say that $w$ verifies Property $(P)$.
\end{defn}

\begin{lem}\label{lem:P_i+t_i->sigma^(i)}
If the sorting word encoding a sorting process of $\sigma$ verifies Property $(P_i)$, 
then at time $t_i$ the elements currently in the stacks are exactly those of $\sigma^{(i)}$.
\end{lem}

\begin{proof}
By definition of time $t_i$ (just before $\sigma_{k_{i}}$ enters the stacks) each element in the stacks has an index smaller than $k_i$.
Moreover among elements of index smaller than $k_i$,
those of value greater than $\sigma_{k_{i}}$ cannot have been output by definition of a sorting,
and those of value smaller than $\sigma_{k_{i}}$ have already been output since $w$ satisfies item $(ii)$ of Property $(P_i)$.
\end{proof}

\begin{lem}\label{lem:i+ii+iii}
Let $w$ be a sorting word for a permutation $\sigma$, $r$ be the number of RTL-minima of $\sigma$ and $\ell \in [1..r]$.
If $w$ verifies $(P_{i})$ for $i \in [1 .. \ell-1]$ then there exists a sorting word $w'$ for $\sigma$ 
that verifies $(P_{i})$ for $i \in [1 .. \ell]$.
\end{lem}

\begin{proof}
Consider the sorting process of $\sigma$ encoded by $w$.
The key idea is to prove that the smallest elements are at the top of the stacks so that we can transform the word $w$ thanks to Lemma \ref{lem:petit}. 

Property ({\romannumeral 2}) for $(P_{\ell })$ states that $\mu_{\sigma_{j}}$ should appear before $\rho_{\sigma_{k_{\ell}}}$ for all $\sigma_{j} < \sigma_{k_{\ell}}$. 
Suppose that there still exists an element $\sigma_{j}$ with $\sigma_{j} < \sigma_{k_{\ell}}$ in the stacks just before $\sigma_{k_{\ell}}$ is pushed into the stacks.
We prove that this element can be popped out before  $\sigma_{k_{\ell}}$ is pushed. 
Let $\sigma_{j_{0}}$ be the smallest element still in the stacks just before $\rho_{\sigma_{k_i}}$. 
By definition, elements smaller than $\sigma_{j_{0}}$ have already been output. 
Consider interval $I = [\sigma_{j_{0}},\sigma_{k_{\ell}}-1]$. 
Those elements are still in the stacks. If they are at the top of the stacks they can be output using Lemma \ref{lem:petit}. 
If not, there exists in the stacks an element $x \notin I$ above an element $y \in I$.
As $V$ is decreasing, those elements are in $H$. 
Moreover $x > \sigma_{k_{\ell}} > y$. 
Then $\sigma_{k_{\ell}}$ cannot be pushed as it will create a pattern $132$ in $H$ with elements $x$ and $y$.
Thus $I$ is at the top of the stacks and we can output it before $\sigma_{k_{\ell}}$ is pushed onto $H$:
using Lemma~\ref{lem:petit}, we build from $w$ a sorting word $w^{(1)}$ for $\sigma$ satisfying $(P_{i})$ for $i \in [1 .. \ell-1]$ 
and Property ({\romannumeral 2}) of $(P_{\ell})$.
This means that $w^{(1)}$ can be decomposed as $w^{(1)} = u \rho_{\sigma_{k_{\ell}}} v$ 
such that the stack configuration $c_{\sigma}(u)$ respects the following constraint:
elements $1, \ldots, \sigma_{k_{\ell}}$ are not in the stacks.

So if we consider the stack configuration $c_{\sigma}(u \rho_{\sigma_{k_{\ell}}} )$, 
element $\sigma_{k_{\ell}}$ is at the top of $H$ and since
$out_{c_{\sigma}(u \rho_{\sigma_{k_{\ell}}} )} (\sigma_{k_{\ell}}) =  \lambda_{\sigma_{k_{\ell}}}  \mu_{\sigma_{k_{\ell}}} $
we can use Lemma \ref{lem:petit} to change the sorting word $w^{(1)}$ into a sorting word 
$w^{(2)}=u  \rho_{\sigma_{k_{\ell}}}  \lambda_{\sigma_{k_{\ell}}}  \mu_{\sigma_{k_{\ell}}} v'$,
satisfying Property ({\romannumeral 1}) for $(P_{\ell})$.

Now we show considering the stack configuration 
$c = c_{\sigma}({u \rho_{\sigma_{k_{\ell}}}  \lambda_{\sigma_{k_{\ell}}}  \mu_{\sigma_{k_{\ell}}} })$
how to transform the word $w^{(2)}$ into a word 
$w' = u \rho_{\sigma_{k_{\ell}}} \lambda_{\sigma_{k_{\ell}}} \mu_{\sigma_{k_{\ell}}} v^{(1)} v^{(2)}$
with $v^{(1)} = out_c( B_{p^{(\ell)}+1}^{(\ell)} \cup \dots \cup B_{s_{\ell}}^{(\ell)})$.
This will conclude the proof.

Notice that elements of $c$ are exactly those of $\sigma^{(\ell)}$ 
since the last operations performed are $\rho_{\sigma_{k_{\ell}}}  \lambda_{\sigma_{k_{\ell}}}  \mu_{\sigma_{k_{\ell}}}$
and elements are pushed in the stacks in increasing order of indices and output in increasing order of values.
Thus
$out_c( B_{p^{(\ell)}+1}^{(\ell)} \cup \dots \cup B_{s_{\ell}}^{(\ell)}) = out(B_{s_{\ell}}^{(\ell)}) \dots out(B_{p^{(\ell)}+1}^{(\ell)})$
(see Lemma~\ref{lem:oneUponTheOther}).
We show by induction on $j$ from $s_{\ell}$ to $p^{(i)}+1$ that we can build a sorting word for $\sigma$ of the form
$u \rho_{\sigma_{k_{\ell}}} \lambda_{\sigma_{k_{\ell}}} \mu_{\sigma_{k_{\ell}}} v^{(1,j)} v^{(2,j)}$ 
with $v^{(1,j)} = out(B_{s_{\ell}}^{(\ell)}) \dots out(B_{j}^{(\ell)})$.
For $j=s_{\ell}$ that is a word in which elements of block $B_{s_{\ell}}$ are output immediately after $\sigma_{k_{\ell}}$ has been output.
By definition of $s_{\ell}$ and because elements of $c$ are exactly those of $\sigma^{(\ell)}$,
all elements of $B_{s_{\ell}}$ lie in the stacks in configuration $c$, are the smallest elements in this configuration 
and lie at the top of the stacks in configuration $c$ (see Lemma~\ref{lem:oneUponTheOther}).
Hence, using Lemma \ref{lem:petit}, there exist a sorting word $w^{(3)}$ for $\sigma$ such that 
$w^{(3)} = u \rho_{\sigma_{k_{\ell}}}  \lambda_{\sigma_{k_{\ell}}}  \mu_{\sigma_{k_{\ell}}} out(B_{s_{\ell}}) v''$. 
Repeating this operation for all blocks $B_{j}$ with $j$ from $s_{\ell} - 1$ to $p^{(i)}+1$, we have Property ({\romannumeral 3}).
\end{proof}

Notice that Property $(P_0)$ is an empty property satisfied by any \sortingw.
Using recursively Lemma~\ref{lem:i+ii+iii} we can transform any sorting word into a sorting word satisfying Property $(P)$,
leading with Lemma~\ref{lem:P_i+t_i->sigma^(i)} to the following theorem:% anglais: recursively utilisé à bon escient ?

\begin{thm}\label{theo:propP}
If $\sigma$ is $2$-stack sortable then there exists a sorting word of $\sigma$ respecting Property $(P)$. 
In particular, in the sorting process that this word encodes, 
the elements currently in the stacks at time $t_i$ are exactly those of $\sigma^{(i)}$.
\end{thm}
% TODO : Aucune reference a ce resultat pourtant crucial -> a ajouter.

Theorem~\ref{theo:propP} ensures that if a permutation is sortable then there exists a sorting in which 
at each time step $t_{i}$, elements in the stacks are exactly those of $\sigma^{(i)}$.
Thus stack configurations at time $t_{i}$ and $t_{i+1}$ satisfy hypothesis of Lemma~\ref{lem:compacc}
and we can apply it to decide if a permutation if \sortable.

\section{An iterative algorithm}\label{sec:graph}

\subsection{A fisrt naïve algorithm}

From Theorem~\ref{theo:propP} a permutation $\sigma$ is \sortable if and only if it admits a sorting process satisfying Property $(P)$.
The main idea is to compute the set of sorting processes of $\sigma$ satisfying Property $(P)$
and decide whether $\sigma$ is \sortable by testing its emptiness.

Verifying $(P)$ means verifying $(P_j)$ for all $j$ from $1$ to $r$, $r$ being the number of right-to-left minima 
(whose indices are denoted $k_{j}$).
The algorithm proceeds in $r$ steps: for all $i$ from $1$ to $r$ 
we iteratively compute the sorting processes of $\sigma_{\leq k_{i}}$ verifying $(P_\ell)$ for all $\ell$ from $1$ to $i$.
As $\sigma_{\leq k_{r}} = \sigma$, the last step gives sorting processes of $\sigma$ satisfying Property $(P)$.

By ``compute the sorting processes of $\sigma_{\leq k_{i}}$'' 
we mean compute the stack configuration just before $\sigma_{k_{i}}$ enters the stacks in such a sorting process.
Note that this is also the stack configuration just after $\sigma_{k_{i}}$ has been output
since $\rho_{\sigma_{k_{i}}} \lambda_{\sigma_{k_{i}}} \mu_{\sigma_{k_{i}}}$ is a factor of any word verifying $(P)$.

\begin{defn}\label{def:configP_i}
We call $P_i$-stack configuration of $\sigma$ a stack configuration $c_{\sigma}(w)$ for which there exists $u$ such that 
the first letter of $u$ is $\rho_{\sigma_{k_i}}$ and $wu$ is a sorting word of $\sigma_{\leq k_{i}}$ verifying $(P)$ for $\sigma_{\leq k_{i}}$
(that is, verifying $(P_\ell)$ for all $\ell$ from $1$ to $i$).
\end{defn}

\begin{lem}\label{lem:sigma_<k_iSortable}
For any $i$ from $1$ to $r$, $\sigma_{\leq k_{i}}$ is \sortable if and only if the set of $P_i$-stack configurations of $\sigma$ is nonempty.
In particular, $\sigma$ is \sortable if and only if the set of $P_r$-stack configurations of $\sigma$ is nonempty.
\end{lem}

\begin{proof}
This is a direct consequence of Definition~\ref{def:configP_i} and Theorem~\ref{theo:propP}.
\end{proof}

\begin{lem}\label{lem:configP_i}
Any $P_i$-stack configuration of $\sigma$ is a pushall stack configuration of $\sigma^{(i)}$
accessible from some $P_{i-1}$-stack configurations of $\sigma$.
\end{lem}

\begin{proof}
By definition of $(P)$, each $P_i$-stack configurations of $\sigma$ is accessible from some 
$P_{i-1}$-stack configurations of $\sigma$ (take the prefix of $w$ that ends just before $\rho_{\sigma_{k_{i-1}}}$).
Moreover it is a pushall stack configuration of $\sigma^{(i)}$ from Lemma~\ref{lem:P_i+t_i->sigma^(i)}.
\end{proof}

% At step $i$, we compute all $P_i$-stack configurations of $\sigma$.
% From Lemma~\ref{lem:P_i+t_i->sigma^(i)}, such a stack configuration is a pushall stack configuration of $\sigma^{(i)}$.
% Moreover by definition of $P_i$, each $P_i$-stack configurations of $\sigma$ is accessible from some 
% $P_{i-1}$-stack configurations of $\sigma$ (take the prefix of $w$ that ends with $\rho_{\sigma_{k_{i-1}}}$.

As explained above, the algorithm proceeds in $r$ steps such that
after step $i$ we know every $P_i$-stack configuration of $\sigma$
and we want to compute the $P_{i+1}$-stack configurations of $\sigma$ at step $i+1$. 
As configurations for $i+1$ are a subset of pushall stack configurations of $\sigma^{(i+1)}$,
a possible algorithm is to take every pair of configurations $(c,c')$ 
with $c$ being a $P_i$-stack configuration of $\sigma$ (computed at step $i$)
and $c'$ be any pushall stack configuration of $\sigma^{(i+1)}$ (given by Algorithm 5 of \cite{PR13}). 
Then we can use Algorithm~\ref{algo:compatibleRectangle} to decide whether $c'$ is accessible from $c$ for $\sigma$.
This leads to the following algorithm deciding whether a permutation $\sigma$ is \sortable:

\begin{algorithm}[H]
% \DontPrintSemicolon
\KwData{$\sigma$ a permutation}
\KwResult{{\tt true} or {\tt false} depending whether $\sigma$ is 2-stack sortable}
\Begin{
$E,F$ two empty sets\;
$E \leftarrow$ PushallConfigs($\sigma^{(1)}$)\;
\For{$i$ from $2$ to $r$}{
$F \leftarrow \varnothing$\;
\For{$c$ in $E$}{
        \For{$c'$ in PushallConfigs($\sigma^{(i)}$)}{
                \If{isAccessible($(c,k_i),(c',k_{i+1}),\sigma$)}{
                        $F \leftarrow F \cup {c'}$\;
                }
}
}
$E \leftarrow F$\;
}
\eIf{$E$ is empty}{
\Return {\tt false}\;
}{
\Return {\tt true}\;
}
}
\caption{$isSortableNaive$}
\label{algo:naif}
\end{algorithm}

Notice that at step $i$, the set $E$ computed contains all $P_i$-stack configurations of $\sigma$ but may contain some other configurations.
However since each configuration of $E$ is a pushall configuration of $\sigma^{(i)}$ 
and is accessible for $\sigma$ from some pushall configurations of $\sigma^{(i-1)}$,
each configuration of $E$ indeed corresponds to some sorting procedure of $\sigma_{\leq k_{i}}$,
proving the correctness of Algorithm~\ref{algo:naif}.

But this algorithm is not polynomial.
Indeed the number of $P_i$-stack configurations of $\sigma$ is possibly exponential. 
However this set can be described by a polynomial representation as a graph $\G{i}$ 
and we can adapt Algorithm~\ref{algo:naif} to obtain a polynomial algorithm.
In this adapted algorithm, the set $E$ computed at step $i$ is exactly the set of $P_i$-stack configurations of $\sigma$.

\subsection{Towards the sorting graph}

We now explain how to adapt Algorithm~\ref{algo:naif} to obtain a polynomial algorithm.
Instead of computing all $P_i$-stack configurations of $\sigma$ (which are pushall stack configurations of $\sigma^{(i)}$), 
we compute the restriction of such configurations to blocks $B_j^{(i)}$ of the $\ominus$-decomposition of $\sigma^{(i)}$.
By Lemma \ref{lem:oneUponTheOther}, those configurations are stacked one upon the others. 
The stack configurations of any block $B_j^{(i)}$ are labeled with an integer which is assigned when the configuration is computed. 
Those pairs configurations / integer will be the vertices of the graph $\G{i}$ which we call a {\em sorting graph}, 
the edges of which representing the configurations that can be stacked one upon the other. 
Vertices of the graph $\G{i}$ are partitioned into levels corresponding to blocks $B_j^{(i)}$.
To ensure the polynomiality of the representation, we will prove that a given integer label could only appear once per level of the graph $\G{i}$. 
As those numbers are assigned to configurations when they are created, each integer corresponding to a pushall stack configuration, 
from \cite{PR13} there exists only a polynomial number of disctincts integers thus of vertices. 
This will be explained in details in the next section. 
The integer indeed can be seen as the memory of the configuration that encodes its history since it has been created:
two configurations which have the same label come from the same initial pushall configuration.

More precisely  a sorting graph $\G{i}$ for a permutation $\sigma$ of size $n$ and an index $i$ verifies the following properties:
 \begin{itemize}\label{def:sortingGraph}
 \item Vertices of $\G{i}$ are partitioned into $s_{i}$ subsets $V_{j}^{(i)}$ with $j \in [1\ldots s_{i}]$.
 \item For any $j \in [1\ldots s_{i}], |V_{j}^{(i)}| \leq 9 n + 2$.
 \item Each vertex $v \in \G{i}$ is a pair \gConf{c}{\ell} with $c$ a stack configuration and \gLabel{\ell} an index called {\em configuration index}.
 \item All configuration indices are distinct inside a graph level $V_{j}^{(i)}$
 \item \gConf{c}{\ell}$\in V_{j}^{(i)} \Rightarrow c$ is a pushall stack configuration of $B_{j}^{(i)}$ accessible for $\sigma$.
 \item There are  edges only between adjacent blocks $V_{j}^{(i)}$, $V_{j+1}^{(i)}$.
 \item Paths between vertices of $V_{1}^{(i)}$ and $V_{s_{i}}^{(i)}$ corresponds to stack configurations of $\sigma^{(i)}$.
More precisely such paths are in one-to-one correspondence with $P_i$-stack configurations of $\sigma$ 
(that is, stack configurations corresponding to a sorting of $\sigma_{\leq k_{i}}$ respecting $(P)$ just before $\sigma_{k_{i}}$ is pushed onto $H$).
 \item For any vertex $v$ of $\G{i}$, there is a path between vertices of $V_{1}^{(i)}$ and $V_{s_{i}}^{(i)}$ going through $v$.
\end{itemize}

Though the definition of sorting graph is complex, its use will be quite understandable and easy. 
Look for example at the permutation $\sigma = 4 3 2 1$. 
There is only one right to left minimum which is $1$.
Compute all possible stack configurations just after $1$ enters $H$.
At this time, all elements are in the stacks since the first element which must be output is $1$. 
More formally, we are looking at the pushall stack configurations of $\sigma$ with $1$ in $H$.

There are $8$ different such configurations which are:

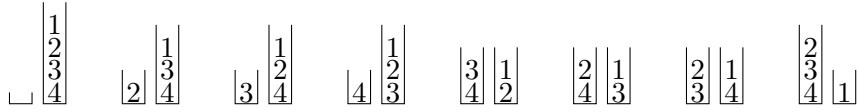
\begin{center}
\begin{tikzpicture}[scale=.3]
\begin{scope}[shift={(0,0)}]
\ssConf{4,3,2,1}{}
\end{scope}
\begin{scope}[shift={(5,0)}]
\ssConf{4,3,1}{2}
\end{scope}
\begin{scope}[shift={(10,0)}]
\ssConf{4,2,1}{3}
\end{scope}
\begin{scope}[shift={(15,0)}]
\ssConf{3,2,1}{4}
\end{scope}
\begin{scope}[shift={(20,0)}]
\ssConf{2,1}{4,3}
\end{scope}
\begin{scope}[shift={(25,0)}]
\ssConf{3,1}{4,2}
\end{scope}
\begin{scope}[shift={(30,0)}]
\ssConf{4,1}{3,2}
\end{scope}
\begin{scope}[shift={(35,0)}]
\ssConf{1}{4,3,2}
\end{scope}
\end{tikzpicture}
\end{center}

The $\ominus$-decomposition of $\sigma$ is $\sigma = \ominus[4,3,2,1]$.
We build a graph with $4$ levels, each level corresponding to pushall stack configurations of a block.

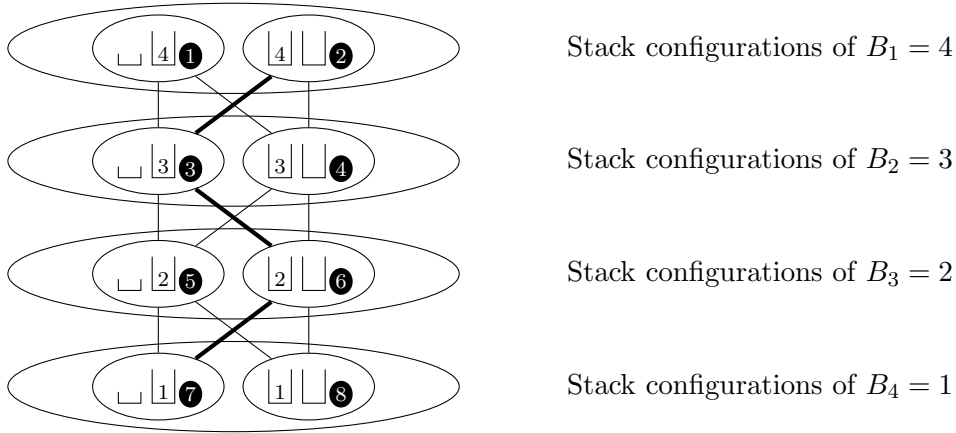
\begin{figure}[H]
\begin{center}
\begin{tikzpicture}
\begin{scope}[shift={(2,0)}]
\ssConfLab{1}{}{c1}{7};
\end{scope}
\begin{scope}[shift={(4,0)}]
\ssConfLab{}{1}{c2}{8};
\end{scope}
\begin{scope}[shift={(2,1.5)}]
\ssConfLab{2}{}{c3}{5};
\end{scope}
\begin{scope}[shift={(4,1.5)}]
\ssConfLab{}{2}{c4}{6};
\end{scope}
\begin{scope}[shift={(2,3)}]
\ssConfLab{3}{}{c5}{3};
\end{scope}
\begin{scope}[shift={(4,3)}]
\ssConfLab{}{3}{c6}{4};
\end{scope}
\begin{scope}[shift={(2,4.5)}]
\ssConfLab{4}{}{c7}{1};
\end{scope}
\begin{scope}[shift={(4,4.5)}]
\ssConfLab{}{4}{c8}{2};
\end{scope}
\draw (c1) -- (c3);
\draw[ultra thick] (c1) -- (c4);
\draw (c2) -- (c3);
\draw (c2) -- (c4);
\draw (c5) -- (c3);
\draw[ultra thick] (c5) -- (c4);
\draw (c6) -- (c3);
\draw (c6) -- (c4);
\draw (c5) -- (c7);
\draw[ultra thick] (c5) -- (c8);
\draw (c6) -- (c7);
\draw (c6) -- (c8);
\draw(3,0) ellipse (3 and .6) node [shift={(7,0)}] {Stack configurations of $B_{4} = 1$}; 
\draw (3,1.5) ellipse (3 and .6) node [shift={(7,0)}] {Stack configurations of $B_{3} = 2$}; ;
\draw (3,3) ellipse (3 and .6) node [shift={(7,0)}] {Stack configurations of $B_{2} = 3$}; ;
\draw (3,4.5) ellipse (3 and .6) node [shift={(7,0)}] {Stack configurations of $B_{1} = 4$}; ;
\end{tikzpicture}
\caption{Graph encoding pushall stack configurations of $\sigma = 4321$.}
\label{fig:graphe}
\end{center}
\end{figure}

Then the $8$ configurations are found taking each of the $8$ different paths going from any configuration of $B_{1}$ 
to configuration \tikz[scale=.3,baseline]\ssConf{1}{}; of $B_{4}$. 
In Figure \ref{fig:graphe}, the thick path gives the stack configuration \tikz[scale=.3,baseline]\ssConf{3,1}{4,2}; 
by stacking the selected configuration of $B_{4}$ above the configuration of $B_{3}$ and so on.

But in the last level $B_{4}$ we only consider configuration \tikz[scale=.3,baseline]\ssConf{1}{};
so this level is useless.
The sorting graph \G{1} for $\sigma = 4 3 2 1$ encodes pushall stack configurations of $\sigma^{(1)} = 432$,
corresponding to stack configurations just {\em before} $1$ enters $H$ (and not after as above).

There are $8$ different such configurations which are:

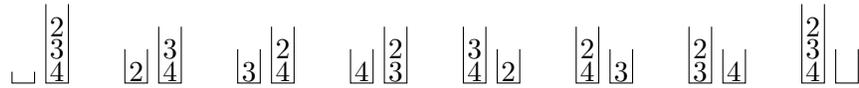
\begin{center}
\begin{tikzpicture}[scale=.3]
\begin{scope}[shift={(0,0)}]
\ssConf{4,3,2}{}
\end{scope}
\begin{scope}[shift={(5,0)}]
\ssConf{4,3}{2}
\end{scope}
\begin{scope}[shift={(10,0)}]
\ssConf{4,2}{3}
\end{scope}
\begin{scope}[shift={(15,0)}]
\ssConf{3,2}{4}
\end{scope}
\begin{scope}[shift={(20,0)}]
\ssConf{2}{4,3}
\end{scope}
\begin{scope}[shift={(25,0)}]
\ssConf{3}{4,2}
\end{scope}
\begin{scope}[shift={(30,0)}]
\ssConf{4}{3,2}
\end{scope}
\begin{scope}[shift={(35,0)}]
\ssConf{}{4,3,2}
\end{scope}
\end{tikzpicture}
\end{center}

As the $\ominus$-decomposition of $\sigma^{(1)}$ is $\sigma^{(1)} = \ominus[4,3,2]$, the sorting graph $\G{1}$ has $3$ levels.

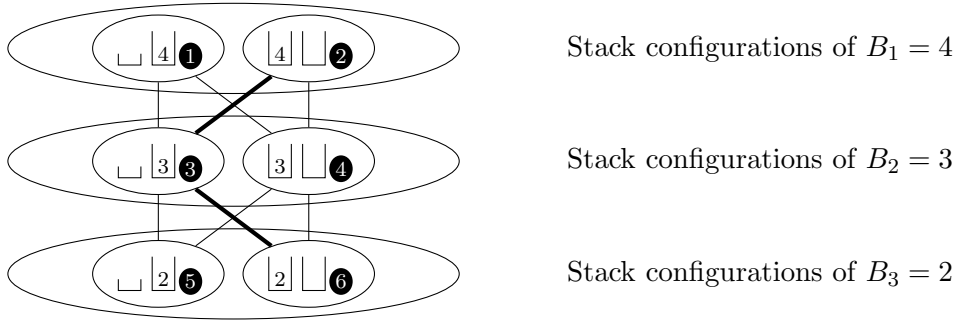
\begin{figure}[H]
\begin{center}
\begin{tikzpicture}
% \begin{scope}[shift={(2,0)}]
% \ssConfLab{1}{}{c1}{1};
% \end{scope}
% \begin{scope}[shift={(4,0)}]
% \ssConfLab{}{1}{c2}{2};
% \end{scope}
\begin{scope}[shift={(2,1.5)}]
\ssConfLab{2}{}{c3}{5};
\end{scope}
\begin{scope}[shift={(4,1.5)}]
\ssConfLab{}{2}{c4}{6};
\end{scope}
\begin{scope}[shift={(2,3)}]
\ssConfLab{3}{}{c5}{3};
\end{scope}
\begin{scope}[shift={(4,3)}]
\ssConfLab{}{3}{c6}{4};
\end{scope}
\begin{scope}[shift={(2,4.5)}]
\ssConfLab{4}{}{c7}{1};
\end{scope}
\begin{scope}[shift={(4,4.5)}]
\ssConfLab{}{4}{c8}{2};
\end{scope}
% \draw (c1) -- (c3);
% \draw[ultra thick] (c1) -- (c4);
% \draw (c2) -- (c3);
% \draw (c2) -- (c4);
\draw (c5) -- (c3);
\draw[ultra thick] (c5) -- (c4);
\draw (c6) -- (c3);
\draw (c6) -- (c4);
\draw (c5) -- (c7);
\draw[ultra thick] (c5) -- (c8);
\draw (c6) -- (c7);
\draw (c6) -- (c8);
% \draw(3,0) ellipse (3 and .6) node [shift={(7,0)}] {Stack configurations of $B_{4} = 1$}; 
\draw (3,1.5) ellipse (3 and .6) node [shift={(7,0)}] {Stack configurations of $B_{3} = 2$}; ;
\draw (3,3) ellipse (3 and .6) node [shift={(7,0)}] {Stack configurations of $B_{2} = 3$}; ;
\draw (3,4.5) ellipse (3 and .6) node [shift={(7,0)}] {Stack configurations of $B_{1} = 4$}; ;
\end{tikzpicture}
\caption{Sorting graph \G{1} of $\sigma = 4321$.}
\label{fig:sortingGraph}
\end{center}
\end{figure}

Then the $8$ configurations are found taking each of the $8$ different paths going from any configuration of $B_{1}$ 
to any configuration of $B_{3}$. 
In Figure \ref{fig:sortingGraph}, the thick path gives the stack configuration \tikz[scale=.3,baseline]\ssConf{3}{4,2}; 
by stacking the selected configuration of $B_{3}$ above the configuration of $B_{2}$ and so on.

We transform Algorithm~\ref{algo:naif} to a polynomial algorithm by computing at step $i$ not all $P_i$-stack configurations of $\sigma$,
but instead the sorting graph $\G{i}$ encoding them.
The graph $\G{i}$ is computed iteratively from the graph $\G{i-1}$ for any $i$ from $2$ to $r$.
The way $\G{i}$ is computed from $\G{i-1}$ depends on the relative values of $p^{(i)}$ and $q^{(i+1)}$.
By definition of a sorting graph given p.\pageref{def:sortingGraph},
if at any step $\G{i}$ is empty, it means that $\sigma_{\leq k_i}$ is not sortable (from Theorem~\ref{theo:propP})
and so is $\sigma$ thus the algorithm returns {\tt false}.
This is summarized in Algorithm \ref{algo:total}.

\begin{algorithm}[H]
% \DontPrintSemicolon
\KwData{$\sigma$ a permutation}
\KwResult{{\tt true} or {\tt false} depending whether $\sigma$ is 2-stack sortable}
\Begin{
${\mathcal G} \leftarrow ComputeG1$\;
\For{$i$ from $2$ to $r$}{
	\eIf{$p^{(i)} = q^{(i+1)}$}{
		${\mathcal G} \leftarrow iteratepEqualsq(\mathcal G)$ or \Return {\tt false}
	}{
	\eIf{$p^{(i)} < q^{(i+1)}$}{
		${\mathcal G} \leftarrow iteratepLessThanq(\mathcal G)$ or \Return {\tt false}
	}{
	${\mathcal G} \leftarrow iteratepGreaterThanq(\mathcal G)$ or \Return {\tt false}
	}
	}
}
\Return {\tt true}
}
\caption{$isSortable$}
 \label{algo:total}
\end{algorithm}

In the next subsections we describe the subprocedures used in our main algorithm $isSortable(\sigma)$.

\subsection{First step: $\G{1}$}

In this subsection, we show how to compute the $P_1$-stack configurations of $\sigma$, that is,
the stack configurations corresponding to time $t_1$ for sorting words of $\sigma_{\leq k_{1}}$ that respect $(P)$ for $\sigma_{\leq k_{1}}$.
% For précisé à cause de (iii) énoncé sur k_i + 1 . A changer ? TODO

From Lemma~\ref{lem:configP_i}, such a stack configuration is a pushall stack configuration of $\sigma^{(1)}$.
Conversely since $\sigma_{k_{1}} = 1$, $\sigma^{(1)}=\sigma_{< k_{1}}$ and
each sorting word of $\sigma_{\leq k_{1}}$ respects $(P_1)$ for $\sigma_{\leq k_{1}}$.
% For précisé à cause de (iii) énoncé sur k_i + 1 . A changer ? TODO
Thus the set of $P_1$-stack configurations of $\sigma$ is the set of pushall stack configurations of $\sigma^{(1)}$.

By Proposition 4.7 of \cite{PR13}, these stack configurations are described by giving 
the set of stack configurations for each block of the $\ominus$-decomposition of $\sigma^{(1)}$. 
More precisely, with $\sigma^{(1)} =  \ominus[B_{1}^{(1)},\ldots, B^{(1)}_{s_{1}}]$ there is a bijection from
$pushallConfigs(B_{1}^{(1)}) \times \dots \times pushallConfigs(B_{s_1}^{(1)})$ onto $pushallConfigs(\sigma^{(1)})$
by stacking configurations one upon the other (as in Lemma~\ref{lem:oneUponTheOther}).
As a consequence, from Lemma~\ref{lem:sigma_<k_iSortable} $\sigma_{\leq k_{1}}$ is not sortable if and only if a set
$pushallConfigs(B_{j}^{(1)})$ is empty.

Moreover it will be useful to label the configurations computed 
so that we attach a distinct integer to each stack configuration when computed.
% created along with the sorting word that leads to the configuration.

At this point, we have encoded all configurations corresponding to words respecting $P$ up to the factor $\rho_{1} \lambda_{1} \mu_{1}$.

The obtained graph is \G{1}. This step is summarized in Algorithm \ref{algo:step1}.

\begin{algorithm}[H]
% \DontPrintSemicolon
\KwData{$\sigma$ a permutation, $num$ a global integer variable}
\KwResult{{\tt false} if $\sigma_{\leq k_{1}}$ is not sortable, the sorting graph $\G{1}$ otherwise.}
\Begin{
% $num \leftarrow 0$\;
$E = \varnothing$\;
Compute $\sigma^{(1)}$ and its $\ominus$-decomposition $\ominus[B_{1}^{(1)},\ldots, B^{(1)}_{s_{1}}]$\;
\For{$j$ from  $1$ to $s_{1}^{(1)}$}{
	$V_{j}^{(1)} \leftarrow \varnothing$\;
	$S = pushallConfigs(B_{j}^{(1)})$\;
	\eIf{$S = \varnothing$}{
		\Return {\tt false}\;
	}{
		\For{$s \in S$}{
			$V_{j}^{(1)} \leftarrow V_{j}^{(1)} \bigcup \{ \gConf{s}{num}\}$\;
			$num \leftarrow num+1$\;
		}
		\If{$j > 1$}{
			$E = E \bigcup  \{(s,s'), s \in V_{j}^{(1)} , s' \in V_{j-1}^{(1)}\}$
		}
	}
}
\Return $\G{1} = (\underset{j \in [1 .. s_{1}^{(1)}]}{\bigcup} V_{j}^{(1)} , E)$
}
\caption{ComputeG1\label{algo:step1}}
\end{algorithm}

\subsection{From step $i$ to step $i+1$}

% Lemma~\ref{lem:P_i+t_i->sigma^(i)} ensures that at time $t_i$, 
% the elements currently in the stacks are exactly those of $\sigma^{(i)}$ so that we can apply Lemma~\ref{lem:compacc}. 
After step $i$ we know the graph $\G{i}$ encoding every $P_i$-stack configuration of $\sigma$
and we want to compute the graph $\G{i+1}$ encoding $P_{i+1}$-stack configurations of $\sigma$ at step $i+1$. 
From Lemma~\ref{lem:configP_i}
we have to check the accessibility of pushall stack configuration of $\sigma^{(i+1)}$ from $P_i$-stack configurations of $\sigma$.
We want to avoid to check every pair of configurations $(c,c')$ 
with $c$ being a $P_i$-stack configuration and $c'$ be a pushall stack configuration of $\sigma^{(i+1)}$
because the number of such pair of configurations is possibly exponential. 
Thus our algorithm focuses not on stack configurations of some $\sigma^{(\ell)}$
but on sets of stack configurations of blocks $B_j^{(\ell)}$,
making use of Lemma~\ref{lem:compacc}. 
Using Lemma~\ref{lem:configP_i}, Lemma~\ref{lem:compacc} can be rephrased as:

\begin{lem}\label{lem:evolution}
Let $c'$ be a total stack configuration of $\sigma^{(i+1)}$, $p = p^{(i)}$ and $q = q^{(i+1)}$.
Then $c'$ is a $P_{i+1}$-stack configuration of $\sigma$ if and only if:
\begin{itemize}
\item For any $j \leq q$, $c'_{|B_{j}^{(i+1)}}$ is a pushall stack configuration of $\sigma_{|B_{j}^{(i+1)}}$, and
\item There exists a $P_i$-stack configuration $c$ of $\sigma$ such that :
\begin{itemize}
\item $c'_{|B^{(i)}_{\min(p,q)} \cup \dots \cup  B^{(i)}_{q}}$ is accessible from $c_{|B^{(i+1)}_{\min(p,q)} \cup \dots \cup B^{(i+1)}_{p}}$ 
for $\sigma_{|B^{(i)}_{p} \bigcup B^{(i+1)}_{q}}$ and
\item $c'_{|B_{1}^{(i+1)} \cup \dots \cup B_{\min(p,q)-1}^{(i+1)}} = c_{|B_{1}^{(i)} \cup \dots \cup B_{\min(p,q)-1}^{(i)}}$
\end{itemize}
\end{itemize}
\end{lem}

Recall that a $P_i$-stack configuration of $\sigma$ is encoded by a path in the sorting graph \G{i}, 
corresponding to the $\ominus$-decomposition of the permutation $\sigma^{(i)}$ into blocks $B_{j}^{(i)}$.
The last point of Lemma \ref{lem:evolution} ensures that the first levels ($1$ to $min(p^{(i)},q^{(i+1)})-1$) 
are the same in \G{i+1} than in \G{i}.
The first point of Lemma \ref{lem:evolution} ensures that the last levels ($>q^{(i+1)}$) of \G{i+1} 
form a complete graph whose vertices are all pushall stack configurations of corresponding blocks. 
So the only unknown levels for \G{i+1} are those between $\min(p^{(i)},q^{(i+1)})$ and $q^{(i+1)}$ and we can compute them by testing accessibility.

There are differents cases depending on the relative values of $p^{(i)}$ and $q^{(i+1)}$.
To lighten the notations in the following, we sometimes write $p$ (resp.~$q$) instead of $p^{(i)}$ (resp.~$q^{(i+1)}$).

\subsubsection{Case $p^{(i)} = q^{(i+1)}$}

If  $p^{(i)} = q^{(i+1)}$ then $B^{(i+1)}_{q^{(i+1)}} \cap A^{(i)} = B^{(i)}_{p^{(i)}} \cap A^{(i)}$ (see Figure~\ref{fig:casep=q}).

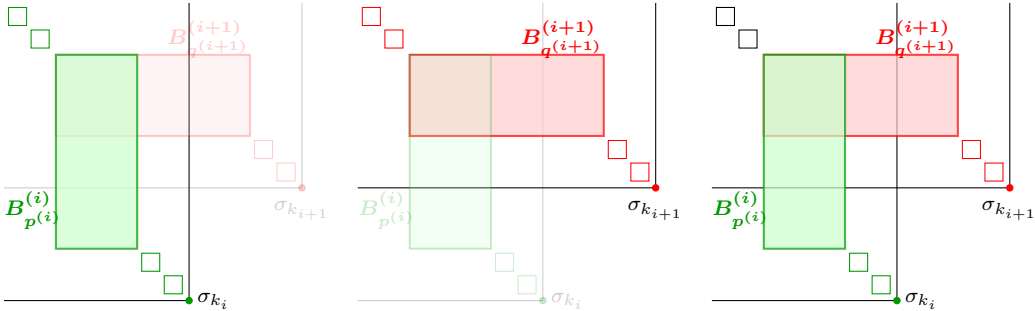
\begin{figure}[H]
\begin{center}
\begin{tikzpicture}[scale=.3]
%\useasboundingbox (0,-1.5) rectangle (4,4);
\draw (-0.2,0) -- (8,0);
\draw [opacity=.2] (-0.2,5) -- (13,5);
\draw (8,0) -- (8,13.2);
\draw [opacity=.2] (13,5) -- (13,13.2);
\draw [green!60!black] (8,0) [fill] circle (4pt);
\draw (9.1,0) node {{\scriptsize $\sigma_{k_i}$}};
\draw [red,opacity=.2] (13,5) [fill] circle (4pt);
\draw [opacity=.2] (13,4) node {{\scriptsize $\sigma_{k_{i+1}}$}};
\draw [green!60!black] (6.9,1.1) rectangle +(0.8,-0.8);
\draw [green!60!black] (5.9,2.1) rectangle +(0.8,-0.8);
\draw [red,opacity=.2] (11.9,6.1) rectangle +(0.8,-0.8);
\draw [red,opacity=.2] (10.9,7.1) rectangle +(0.8,-0.8);
%\draw (4.5,8.5) rectangle +(0.8,-0.8);
%\draw (3.5,9.5) rectangle +(0.8,-0.8);
%\draw (2.5,10.5) rectangle +(0.8,-0.8);
\draw [green!60!black] (1,12) rectangle +(0.8,-0.8);
\draw [green!60!black] (0,13) rectangle +(0.8,-0.8);
\draw[thick, red,fill=red!20,opacity=.2] (10.7,7.3) rectangle +(-8.6,3.6);
\draw[red,opacity=.2] (8.85,11.6) node {{\scriptsize \boldmath$B^{(i+1)}_{q^{(i+1)}}$}};
\draw[thick, green!60!black,fill=green!20,opacity=.7] (5.7,2.3) rectangle +(-3.6,8.6);
\draw[green!60!black] (1.1,4) node {{\scriptsize \boldmath$B^{(i)}_{p^{(i)}}$}};
\end{tikzpicture}
\begin{tikzpicture}[scale=.3]
%\useasboundingbox (0,-1.5) rectangle (4,4);
\draw [opacity=.2] (-0.2,0) -- (8,0);
\draw  (-0.2,5) -- (13,5);
\draw [opacity=.2] (8,0) -- (8,13.2);
\draw (13,5) -- (13,13.2);
\draw [green!60!black,opacity=.2] (8,0) [fill] circle (4pt);
\draw [opacity=.2] (9.1,0) node {{\scriptsize $\sigma_{k_i}$}};
\draw [red] (13,5) [fill] circle (4pt);
\draw (13,4) node {{\scriptsize $\sigma_{k_{i+1}}$}};
\draw [green!60!black,opacity=.2] (6.9,1.1) rectangle +(0.8,-0.8);
\draw [green!60!black,opacity=.2] (5.9,2.1) rectangle +(0.8,-0.8);
\draw [red] (11.9,6.1) rectangle +(0.8,-0.8);
\draw [red] (10.9,7.1) rectangle +(0.8,-0.8);
%\draw (4.5,8.5) rectangle +(0.8,-0.8);
%\draw (3.5,9.5) rectangle +(0.8,-0.8);
%\draw (2.5,10.5) rectangle +(0.8,-0.8);
\draw [red] (1,12) rectangle +(0.8,-0.8);
\draw [red] (0,13) rectangle +(0.8,-0.8);
\draw[thick, red,fill=red!20,opacity=.7] (10.7,7.3) rectangle +(-8.6,3.6);
\draw[red] (8.85,11.6) node {{\scriptsize \boldmath$B^{(i+1)}_{q^{(i+1)}}$}};
\draw[thick, green!60!black,fill=green!20,opacity=.2] (5.7,2.3) rectangle +(-3.6,8.6);
\draw[green!60!black,opacity=.2] (1.1,4) node {{\scriptsize \boldmath$B^{(i)}_{p^{(i)}}$}};
\end{tikzpicture}
\begin{tikzpicture}[scale=.3]
%\useasboundingbox (0,-1.5) rectangle (4,4);
\draw (-0.2,0) -- (8,0);
\draw (-0.2,5) -- (13,5);
\draw (8,0) -- (8,13.2);
\draw (13,5) -- (13,13.2);
\draw [green!60!black] (8,0) [fill] circle (4pt);
\draw (9.1,0) node {{\scriptsize $\sigma_{k_i}$}};
\draw [red] (13,5) [fill] circle (4pt);
\draw (13,4) node {{\scriptsize $\sigma_{k_{i+1}}$}};
\draw [green!60!black] (6.9,1.1) rectangle +(0.8,-0.8);
\draw [green!60!black] (5.9,2.1) rectangle +(0.8,-0.8);
\draw [red] (11.9,6.1) rectangle +(0.8,-0.8);
\draw [red] (10.9,7.1) rectangle +(0.8,-0.8);
%\draw (4.5,8.5) rectangle +(0.8,-0.8);
%\draw (3.5,9.5) rectangle +(0.8,-0.8);
%\draw (2.5,10.5) rectangle +(0.8,-0.8);
\draw (1,12) rectangle +(0.8,-0.8);
\draw (0,13) rectangle +(0.8,-0.8);
\draw[thick, red,fill=red!20,opacity=.7] (10.7,7.3) rectangle +(-8.6,3.6);
\draw[red] (8.85,11.6) node {{\scriptsize \boldmath$B^{(i+1)}_{q^{(i+1)}}$}};
\draw[thick, green!60!black,fill=green!20,opacity=.7] (5.7,2.3) rectangle +(-3.6,8.6);
\draw[green!60!black] (1.1,4) node {{\scriptsize \boldmath$B^{(i)}_{p^{(i)}}$}};
\end{tikzpicture}
\caption{Block decomposition of $\sigma^{(i)}$ and of $\sigma^{(i+1)}$ when $p^{(i)}=q^{(i+1)}$}
\label{fig:casep=q}
\end{center}
\end{figure}

We have the sorting graph $\G{i}$ encoding all $P_i$-stack configurations of $\sigma$
and we want to compute the sorting graph $\G{i+1}$ encoding all $P_{i+1}$-stack configurations of $\sigma$ assuming that 
$p^{(i)} = q^{(i+1)} = min(p^{(i)},q^{(i+1)})$.

In this case, from Lemma \ref{lem:evolution} we only have to check accessiblity of pushall configurations of $B_{q}^{(i+1)}$
from configurations of $B^{(i)}_{p}$ belonging to level $p$ of $\G{i}$.
Indeed from the definition of a sorting graph given p.\pageref{def:sortingGraph}, 
for any vertex $v$ of $\G{i}$ there is a path between vertices of $V_{1}^{(i)}$ and $V_{s_{i}}^{(i)}$ going through $v$,
and such a path corresponds to a $P_i$-stack configuations of $\sigma$. 
Thus for any configurations $x$ of $B^{(i)}_{p}$ belonging to a vertex $v$ of level $p$ of $\G{i}$, 
there is at least one $P_i$-stack configurations $c$ of $\sigma$ such that $c_{|{B^{(i)}_{p}}} = x$, and 
$c_{|B_{1}^{(i)} \cup \dots \cup B_{\min(p,q)-1}^{(i)}}$ is encoded by a path from $v$ to level $p$ of $\G{i}$
(which go through each level $< p$).

If there is no pushall configuration of $B_{q}^{(i+1)}$ accessible from 
some configurations of $B^{(i)}_{p}$ belonging to level $p$ of $\G{i}$, or if $\sigma^{(i+1)}$ has no pushall configuration,
then $\sigma$ has no $P_{i+1}$-stack configuration and $\sigma_{\leq k_{i+1}}$ is not sortable (from Lemma~\ref{lem:sigma_<k_iSortable}).

This leads to the following algorithm:
\begin{algorithm}[H]
% \DontPrintSemicolon
\KwData{$\sigma$ a permutation and \G{i} the sorting graph at step $i$}
\KwResult{{\tt false} if $\sigma_{\leq k_{i+1}}$ is not sortable, the sorting graph $\G{i+1}$ otherwise.}
\Begin{
${\mathcal G}$ an empty sorting graph with $s_{i+1}$ levels\;
${\mathcal G}' \leftarrow ComputeG1(\sigma^{(i+1)})$ (pushall sorting graph of $\sigma^{(i+1)}$) or \Return {\tt false}\;
% Copy the $p-1$ first levels of \G{i} into ${\mathcal G}$\;
Copy levels $q+1 \ldots s_{i+1}$ of ${\mathcal G}'$ into the same levels of ${\mathcal G}$\;
\For{$\gConf{c}{\ell}$ in level $p$ of \G{i}}{
	${\mathcal H}$ the subgraph of \G{i} induced by $\gConf{c}{\ell}$ in levels $< p$\;
	\For{$\gConf{c'}{\ell'}$ in level $q$ of ${\mathcal G}'$}{
		\If{$isAccessible(c,c',\sigma_{|B^{(i)}_{p} \bigcup B^{(i+1)}_{q}})$}{
			 Add $\gConf{c'}{\ell'}$ in level $q$ of ${\mathcal G}$ (if not already done)\;
			 Merge ${\mathcal H}$ in levels $\leq q$ of ${\mathcal G}$ with $\gConf{c'}{\ell'}$ as origin\;
% 			 Merge ${\mathcal G}$ with the graph ${\mathcal G}''_{p-1} \times \gConf{c'}{\ell} \times {\mathcal G'_{q+1 }}$ at levels $p-1, p,q+1$.
		}
	}
}
\If{level $q$ of ${\mathcal G}$ is empty}{
	\Return {\tt false}\;
}{
	\For{$\gConf{c'}{\ell'}$ in level $q$ of ${\mathcal G}$}{
	Add all edges from $\gConf{c'}{\ell'}$ to each vertex of level $q+1$ of ${\mathcal G}$;
	}
	\Return ${\mathcal G}$
} 
}
\caption{$iteratepEqualsq(\G{i})$ \label{algo:steppq}}
\end{algorithm}

\subsubsection{Case $p^{(i)} < q^{(i+1)}$}

If  $p^{(i)} < q^{(i+1)}$ then $B^{(i+1)}_{q^{(i+1)}} \cap A^{(i)} \varsubsetneq B^{(i)}_{p^{(i)}} \cap A^{(i)}$ (see Figure~\ref{fig:p<q}).

\begin{figure}[ht!]
\begin{center}
\begin{tikzpicture}[scale=.3]
%\useasboundingbox (0,-1.5) rectangle (4,4);
\draw (-0.2,0) -- (8,0);
\draw [opacity=.2]  (-0.2,5) -- (13,5);
\draw (8,0) -- (8,13.2);
\draw [opacity=.2]  (13,5) -- (13,13.2);
\draw [green!60!black] (8,0) [fill] circle (4pt);
\draw (9.1,0) node {{\scriptsize $\sigma_{k_i}$}};
\draw [opacity=.2]  (13,5) [fill] circle (4pt);
\draw[opacity=.2] (13,4) node {{\scriptsize $\sigma_{k_{i+1}}$}};
\draw[thick, red!20, fill=red!20, opacity=.3] (10.7,7.3) rectangle +(-7.5,2.5);
\draw[red,opacity=.2] (8.85,10.5) node {{\scriptsize \boldmath$B^{(i+1)}_{q^{(i+1)}}$}};
\draw[thick, green!60!black, fill=green!20, opacity=.7] (5.7,2.3) rectangle
+(-4.78,9.78);
\draw[green!60!black] (2.7,1.4) node {{\scriptsize \boldmath$B^{(i)}_{p^{(i)}}$}};
\draw [green!60!black] (6.9,1.1) rectangle +(0.8,-0.8);
\draw [green!60!black] (5.9,2.1) rectangle +(0.8,-0.8);
\draw [opacity=.2] (11.9,6.1) rectangle +(0.8,-0.8);
\draw [opacity=.2] (10.9,7.1) rectangle +(0.8,-0.8);
%\draw (4.5,8.5) rectangle +(0.8,-0.8);
%\draw (3.5,9.5) rectangle +(0.8,-0.8);
%\draw (2.25,10.75) rectangle +(0.8,-0.8);
%\draw (1.25,11.75) rectangle +(0.8,-0.8);
\draw  [green!60!black] (0,13) rectangle +(0.8,-0.8);
\end{tikzpicture}
\begin{tikzpicture}[scale=.3]
%\useasboundingbox (0,-1.5) rectangle (4,4);
\draw [opacity=.2](-0.2,0) -- (8,0);
\draw  (-0.2,5) -- (13,5);
\draw [opacity=.2](8,0) -- (8,13.2);
\draw  (13,5) -- (13,13.2);
\draw[opacity=.2,green!60!black ] (8,0) [fill] circle (4pt);
\draw[opacity=.2] (9.1,0) node {{\scriptsize $\sigma_{k_i}$}};
\draw [red]  (13,5) [fill] circle (4pt);
\draw  (13,4) node {{\scriptsize $\sigma_{k_{i+1}}$}};
\draw[thick, red, fill=red!20, opacity=.7] (10.7,7.3) rectangle +(-7.5,2.5);
\draw[red] (8.85,10.5) node {{\scriptsize \boldmath$B^{(i+1)}_{q^{(i+1)}}$}};
\draw[thick, green!60!black, fill=green!20, opacity=.2] (5.7,2.3) rectangle
+(-4.78,9.78);
\draw[green!60!black,opacity=.2] (2.7,1.4) node {{\scriptsize
\boldmath$B^{(i)}_{p^{(i)}}$}};
\draw [green!60!black,opacity=.2] (6.9,1.1) rectangle +(0.8,-0.8);
\draw [green!60!black, opacity=.2] (5.9,2.1) rectangle +(0.8,-0.8);
\draw  [red]  (11.9,6.1) rectangle +(0.8,-0.8);
\draw  [red] (10.9,7.1) rectangle +(0.8,-0.8);
%\draw (4.5,8.5) rectangle +(0.8,-0.8);
%\draw (3.5,9.5) rectangle +(0.8,-0.8);
\draw [red] (2.25,10.75) rectangle +(0.8,-0.8);
\draw [red] (1.25,11.75) rectangle +(0.8,-0.8);
\draw [red] (0,13) rectangle +(0.8,-0.8);
\end{tikzpicture}
\begin{tikzpicture}[scale=.3]
%\useasboundingbox (0,-1.5) rectangle (4,4);
\draw (-0.2,0) -- (8,0);
\draw  (-0.2,5) -- (13,5);
\draw (8,0) -- (8,13.2);
\draw  (13,5) -- (13,13.2);
\draw[green!60!black ] (8,0) [fill] circle (4pt);
\draw  (9.1,0) node {{\scriptsize $\sigma_{k_i}$}};
\draw [red]  (13,5) [fill] circle (4pt);
\draw  (13,4) node {{\scriptsize $\sigma_{k_{i+1}}$}};
\draw[red] (8.85,10.5) node {{\scriptsize \boldmath$B^{(i+1)}_{q^{(i+1)}}$}};
\draw[thick, green!60!black, fill=green!20,opacity=.7] (5.7,2.3) rectangle
+(-4.78,9.78);
\draw[thick, red, fill=red!20, opacity=.7] (10.7,7.3) rectangle +(-7.5,2.5);
\draw[green!60!black] (2.7,1.4) node {{\scriptsize \boldmath$B^{(i)}_{p^{(i)}}$}};
\draw [green!60!black] (6.9,1.1) rectangle +(0.8,-0.8);
\draw [green!60!black] (5.9,2.1) rectangle +(0.8,-0.8);
\draw  [red]  (11.9,6.1) rectangle +(0.8,-0.8);
\draw  [red] (10.9,7.1) rectangle +(0.8,-0.8);
%\draw (4.5,8.5) rectangle +(0.8,-0.8);
%\draw (3.5,9.5) rectangle +(0.8,-0.8);
\draw [red] (2.25,10.75) rectangle +(0.8,-0.8);
\draw [red] (1.25,11.75) rectangle +(0.8,-0.8);
\draw (0,13) rectangle +(0.8,-0.8);
\end{tikzpicture}
\caption{Block decomposition of $\sigma^{(i)}$ and of $\sigma^{(i+1)}$ when $p^{(i)} < q^{(i+1)}$}
\label{fig:p<q}
\end{center}
\end{figure}
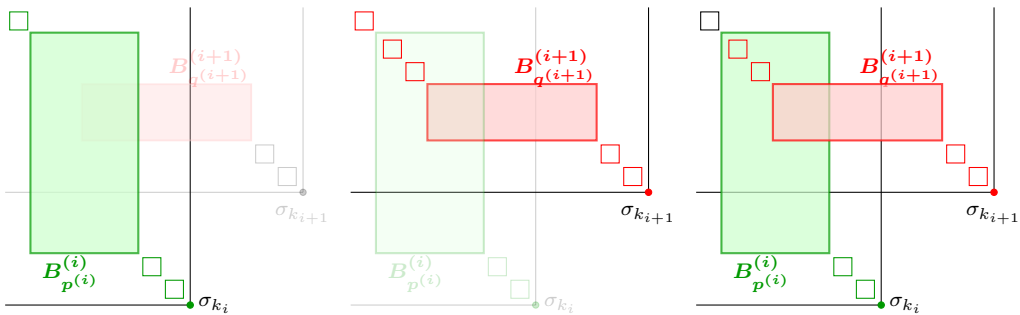

Again, Lemma \ref{lem:evolution} ensures that the first $p-1$ levels of \G{i+1} come from those of \G{i} and 
the levels $> q$ are all pushall stack configurations of the blocks $B^{(i+1)}_{> q}$ of $\sigma^{(i+1)}$. 
The difficult part is from level $p$ to level $q$. 
As in the preceding case, by Lemma \ref{lem:evolution}, 
we have to select among pushall stack configurations of blocks $p,p+1,\ldots, q$ of $\sigma^{(i+1)}$ 
those accessible from a configuration of $B^{(i)}_{p}$ that appears at level $p$ in \G{i}. 
We can restrict the accessibility test from configurations of $B^{(i)}_{p}$ appearing in graph \G{i}
to pushall stack configurations of $B^{(i+1)}_{q}$. 
Indeed, Lemma \ref{lem:HautGaucheA} ensures that elements of blocks $B_{j}^{(i+1)}$ for $j$ from $p$ to $q-1$ 
are in the same stack at time $t_i$ and at time $t_{i+1}$.
Thus configurations of $B_{j}^{(i+1)}$ for $j$ from $p$ to $q-1$ are restrictions of configurations of $B_{p}^{(i)}$.
We keep the same label in the vertex to encode that those configurations of
$B_{p}^{(i+1)},B_{p+1}^{(i+1)},\ldots, B_{q-1}^{(i+1)}$ come from the same configuration of $B_{p}^{(i)}$
and we build edges between vertices of $B_{j+1}^{(i+1)}$ and $B_{j}^{(i+1)}$ that come from the same configuration of $B_{p}^{(i)}$.
It is because of this case $p=q$ that we have to label configurations in our sorting graph.
Indeed two different stack configurations $c_1$ and $c_2$ of $B_{p}^{(i)}$ may have the same restriction to some block $B_{j}^{(i+1)}$
but not be compatible with the same configurations, thus we want the corresponding vertices of level $j$ of \G{i+1} to be distinct,
that's why we use labels.

%More precisely if $c$ a pushall stack configuration of $\sigma^{(i)}$ is compatible with a pushall stack configuration of 
%$\sigma^{(i+1)}$, then  these configurations are equals for the $p-1$ first blocks, 
%block $p$ of $c$ must be compatible with blocks $p, p+1, \ldots q$ of $c'$ and the last blocks are those of $c'$. 
%The difference with case $p=q$ relies on the fact that block $B_{p}^{(i)}$ is divided into several blocks when considering $\sigma^{(i+1)}$. 
%Thus a configuration for this block which is a vertex of \G{i} does not give a vertex in \G{i+1} as in the preceding case but is expanded into a path, 
%the vertices of which being the restrictions of the configuration to the block pieces in $\sigma^{(i+1)}$. 

More precisely we have the following algorithm.
 
%We copy the first $p-1$ levels of graph \G{i} into \G{i+1}. 
%Then, for each configuration \gConf{c'}{\ell'} in level $q$ of  $G'$ that correspond to a possible pushall stack configuration of block $B_{q}^{(i+1)}$ and for each configuration \gConf{c}{\ell} in level $p$ of \G{i}, if $c$ is compatible with $c'$ then 
%
%\begin{itemize}
%\item For all $j$ from $q-1$ down to $p$, add \gConf{c_{|B_{j}^{(i+1)}}}{\ell} to $V^{(i+1)}_{j}$, the $j$-th level of \G{i+1}.
%\item Add \gConf{c'}{\ell'} in $V^{(i+1)}_q$ that is the $q$-th level of \G{i+1} if not exists.
%\item Add edges between \gConf{c'}{\ell'} of $V^{(i+1)}_q$ and \gConf{c_{|B_{q-1}^{(i+1)}}}{\ell} of $V^{(i+1)}_{q-1}$.
%\item For $j$ from $p+1$ to $q-1$, add edges between \gConf{c_{|B_{j}^{(i+1)}}}{\ell} and \gConf{c_{|B_{j-1}^{(i+1)}}}{\ell}.
%\item At last add an edge between $(c_{|B_{p}^{(i+1)}, \ell})$ and vertices $(c'',m)$ at level $p-1$ of \G{i} that are adjacent to $(c,\ell)$.
%\end{itemize}

\begin{algorithm}[H]
% \DontPrintSemicolon
\KwData{$\sigma$ a permutation and \G{i} the sorting graph at step $i$}
\KwResult{{\tt false} if $\sigma_{\leq k_{i+1}}$ is not sortable, the sorting graph $\G{i+1}$ otherwise.}
\Begin{
${\mathcal G}$ an empty sorting graph with $s_{i+1}$ levels\;
${\mathcal G}' \leftarrow ComputeG1(\sigma^{(i+1)})$ (pushall sorting graph of $\sigma^{(i+1)}$) or \Return {\tt false}\;
% Copy the $p-1$ first levels of \G{i} into ${\mathcal G}$\;
Copy levels $q+1, \ldots, s_{i+1}$ of ${\mathcal G}'$ into the same levels of ${\mathcal G}$\;
\For{$\gConf{c}{\ell}$ in level $p$ of \G{i}}{
	${\mathcal H}$ the subgraph of \G{i} induced by $\gConf{c}{\ell}$ in levels $< p$\;
	\For{$\gConf{c'}{\ell'}$ in level $q$ of ${\mathcal G}'$}{
		\If{$isAccessible(c,c',\sigma_{|B^{(i)}_{p} \bigcup B^{(i+1)}_{q}})$}{
			 Add $\gConf{c'}{\ell'}$ in level $q$ of ${\mathcal G}$ (if not already done)\;
			 \For{$j$ from $q-1$ downto $p$}{
			 Add $\gConf{c_{|B_{j}^{(i+1)}}}{\ell}$ in level $j$ of ${\mathcal G}$\;
			 Add an edge between $\gConf{c_{|B_{j}^{(i+1)}}}{\ell}$ and $\gConf{c_{|B_{j+1}^{(i+1)}}}{\ell}$ in ${\mathcal G}$.
			 }
			 Merge ${\mathcal H}$ in levels $\leq p$ of ${\mathcal G}$ with $\gConf{c_{|B_{p}^{(i+1)}}}{\ell}$ as origin\;
% Merge ${\mathcal G}$ with the graph 
% ${\mathcal G}''_{p-1} \times \gConf{c_{|B_{p}^{(i+1)}}}{\ell} \leftrightarrow \gConf{c_{|B_{p+1}^{(i+1)}}}{\ell} \leftrightarrow \ldots \leftrightarrow \gConf{c_{|B_{q-1}^{(i+1)}}}{\ell} \leftrightarrow \gConf{c'}{\ell'} \times {\mathcal G'_{q+1 }}$ 
% at levels $p-1, p,\ldots, q+1$.
		}
	}
}
\If{level $q$ of ${\mathcal G}$ is empty}{
	\Return {\tt false}\;
}{
	\For{$\gConf{c'}{\ell'}$ in level $q$ of ${\mathcal G}$}{
	Add all edges from $\gConf{c'}{\ell'}$ to each vertex of level $q+1$ of ${\mathcal G}$;
	}
}
\Return ${\mathcal G}$\;
}
\caption{$iteratepLessThanq(\G{i})$}
\label{algo:stepp<q}
\end{algorithm}

Note that in Algorithm~\ref{algo:stepp<q}, before calling $isAccessible(c,c',\sigma_{|B^{(i)}_{p} \bigcup B^{(i+1)}_{q}})$
we extend configuration $c'$ to $D^{(i)} \bigcup B_{q}^{(i+1)}$ 
by assigning the same stack than in $c$ to points of $D^{(i)} \setminus B_{q}^{(i+1)}$.
This is justified by Lemma~\ref{lem:HautGaucheA}.

\newpage

\subsubsection{Case $p^{(i)} > q^{(i+1)}$}

If $p^{(i)} > q^{(i+1)}$ then $B^{(i)}_{p^{(i)}} \cap A^{(i)} \varsubsetneq B^{(i+1)}_{q^{(i+1)}} \cap A^{(i)}$ (see Figure~\ref{fig:p>q}).

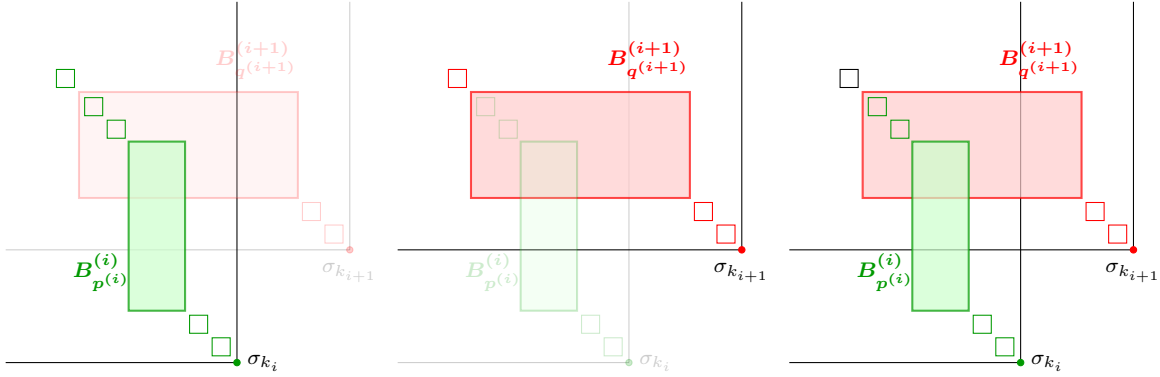
\begin{figure}[ht!]
\begin{center}
\begin{tikzpicture}[scale=.3]
%\useasboundingbox (0,-1.5) rectangle (4,4);
\draw (-2.25,0) -- (8,0);
\draw [opacity=.2] (-2.25,5) -- (13,5);
\draw (8,0) -- (8,16);
\draw [opacity=.2] (13,5) -- (13,16);
\draw [green!60!black](8,0) [fill] circle (4pt);
\draw (9.2,0) node {{\scriptsize $\sigma_{k_i}$}};
\draw [opacity=.2,red] (13,5) [fill] circle (4pt);
\draw [opacity=.2] (13,4) node {{\scriptsize $\sigma_{k_{i+1}}$}};
\draw [green!60!black](6.9,1.1) rectangle +(0.8,-0.8);
\draw [green!60!black] (5.9,2.1) rectangle +(0.8,-0.8);
\draw [red,opacity=.2] (11.9,6.1) rectangle +(0.8,-0.8);
\draw [red,opacity=.2] (10.9,7.1) rectangle +(0.8,-0.8);
%\draw (4.5,8.5) rectangle +(0.8,-0.8);
%\draw (3.5,9.5) rectangle +(0.8,-0.8);
\draw [green!60!black] (0,13) rectangle +(0.8,-0.8);
%\draw (-1.25,14.25) rectangle +(0.8,-0.8);
%\draw (-2.25,15.25) rectangle +(0.8,-0.8);
%\draw (-3.5,16.55) rectangle +(0.8,-0.8);
\draw[thick, red, fill=red!20, opacity=.2]  (10.7,7.3) rectangle +(-9.7,4.7);
\draw[red,opacity=.2] (8.85,13.4) node {{\scriptsize \boldmath$B^{(i+1)}_{q^{(i+1)}}$}};
\draw[thick, green!60!black, fill=green!20,opacity=.7](5.7,2.3) rectangle +(-2.5,7.5);
\draw[green!60!black] (2,4) node {{\scriptsize \boldmath$B^{(i)}_{p^{(i)}}$}};
\draw [green!60!black](2.25,10.75) rectangle +(0.8,-0.8);
\draw [green!60!black](1.25,11.75) rectangle +(0.8,-0.8);
\end{tikzpicture}
\begin{tikzpicture}[scale=.3]
%\useasboundingbox (0,-1.5) rectangle (4,4);
\draw [opacity=.2] (-2.25,0) -- (8,0);
\draw (-2.25,5) -- (13,5);
\draw [opacity=.2] (8,0) -- (8,16);
\draw (13,5) -- (13,16);
\draw [green!60!black,opacity=.2](8,0) [fill] circle (4pt);
\draw [opacity=.2] (9.2,0) node {{\scriptsize $\sigma_{k_i}$}};
\draw [red] (13,5) [fill] circle (4pt);
\draw (13,4) node {{\scriptsize $\sigma_{k_{i+1}}$}};
\draw [green!60!black,opacity=.2](6.9,1.1) rectangle +(0.8,-0.8);
\draw [green!60!black,opacity=.2] (5.9,2.1) rectangle +(0.8,-0.8);
\draw [red] (11.9,6.1) rectangle +(0.8,-0.8);
\draw [red] (10.9,7.1) rectangle +(0.8,-0.8);
%\draw (4.5,8.5) rectangle +(0.8,-0.8);
%\draw (3.5,9.5) rectangle +(0.8,-0.8);
\draw [red](0,13) rectangle +(0.8,-0.8);
%\draw (-1.25,14.25) rectangle +(0.8,-0.8);
%\draw (-2.25,15.25) rectangle +(0.8,-0.8);
%\draw (-3.5,16.55) rectangle +(0.8,-0.8);
\draw[thick, red, fill=red!20, opacity=.7]  (10.7,7.3) rectangle +(-9.7,4.7);
\draw[red] (8.85,13.4) node {{\scriptsize \boldmath$B^{(i+1)}_{q^{(i+1)}}$}};
\draw[thick, green!60!black, fill=green!20,opacity=.2](5.7,2.3) rectangle +(-2.5,7.5);
\draw[green!60!black,opacity=.2] (2,4) node {{\scriptsize \boldmath$B^{(i)}_{p^{(i)}}$}};
\draw [green!60!black,opacity=.2](2.25,10.75) rectangle +(0.8,-0.8);
\draw [green!60!black,opacity=.2](1.25,11.75) rectangle +(0.8,-0.8);
\end{tikzpicture}
\begin{tikzpicture}[scale=.3]
%\useasboundingbox (0,-1.5) rectangle (4,4);
\draw (-2.25,0) -- (8,0);
\draw (-2.25,5) -- (13,5);
\draw (8,0) -- (8,16);
\draw (13,5) -- (13,16);
\draw [green!60!black](8,0) [fill] circle (4pt);
\draw (9.2,0) node {{\scriptsize $\sigma_{k_i}$}};
\draw [red] (13,5) [fill] circle (4pt);
\draw (13,4) node {{\scriptsize $\sigma_{k_{i+1}}$}};
\draw [green!60!black](6.9,1.1) rectangle +(0.8,-0.8);
\draw [green!60!black] (5.9,2.1) rectangle +(0.8,-0.8);
\draw [red] (11.9,6.1) rectangle +(0.8,-0.8);
\draw [red] (10.9,7.1) rectangle +(0.8,-0.8);
%\draw (4.5,8.5) rectangle +(0.8,-0.8);
%\draw (3.5,9.5) rectangle +(0.8,-0.8);
\draw (0,13) rectangle +(0.8,-0.8);
%\draw (-1.25,14.25) rectangle +(0.8,-0.8);
%\draw (-2.25,15.25) rectangle +(0.8,-0.8);
%\draw (-3.5,16.55) rectangle +(0.8,-0.8);
\draw[thick, red, fill=red!20, opacity=.7]  (10.7,7.3) rectangle +(-9.7,4.7);
\draw[red] (8.85,13.4) node {{\scriptsize \boldmath$B^{(i+1)}_{q^{(i+1)}}$}};
\draw[thick, green!60!black, fill=green!20,opacity=.7](5.7,2.3) rectangle +(-2.5,7.5);
\draw[green!60!black] (2,4) node {{\scriptsize \boldmath$B^{(i)}_{p^{(i)}}$}};
\draw [green!60!black](2.25,10.75) rectangle +(0.8,-0.8);
\draw [green!60!black](1.25,11.75) rectangle +(0.8,-0.8);
\end{tikzpicture}
\caption{Block decomposition of $\sigma^{(i)}$ and of $\sigma^{(i+1)}$ when $p^{(i)} > q^{(i+1)}$}
\label{fig:p>q}
\end{center}
\end{figure}

\begin{algorithm}[H]
% \DontPrintSemicolon
\KwData{$\sigma$ a permutation and \G{i} the sorting graph at step $i$}
\KwResult{{\tt false} if $\sigma_{\leq k_{i+1}}$ is not sortable, the sorting graph $\G{i+1}$ otherwise}
\Begin{
${\mathcal G}$ an empty sorting graph with $s_{i+1}$ levels\;
${\mathcal G}' \leftarrow ComputeG1(\sigma^{(i+1)})$ (pushall sorting graph of $\sigma^{(i+1)}$) or \Return {\tt false}\;
% Copy the $q-1$ first levels of \G{i} into ${\mathcal G}$\;
Copy levels $q+1, \ldots, s_{i+1}$ of ${\mathcal G}'$ into the same levels of ${\mathcal G}$\;
\For{$\gConf{c}{\ell}$ in level $p$ of \G{i}}{
	\For{$\gConf{c'}{\ell'}$ in level $q$ of ${\mathcal G}'$}{
		\If{$isAccessible(c,c',\sigma_{|B^{(i)}_{p} \bigcup B^{(i+1)}_{q}})$}{
			\If{there is a path $\gConf{c}{\ell} \leftrightarrow \gConf{c'_{|B_{p-1}^{(i)}}}{\ell_1} \leftrightarrow \ldots \leftrightarrow \gConf{c'_{|B_{q}^{(i)}}}{\ell_k}$ in \G{i}}{
				Add $\gConf{c'}{\ell'}$ in level $q$ of ${\mathcal G}$ (if not already done)\;
				${\mathcal H}$ the subgraph of \G{i} induced by $\gConf{c'_{|B_{q}^{(i)}}}{\ell_k}$ in levels $< q$\;
				Merge ${\mathcal H}$ in levels $\leq q$ of ${\mathcal G}$ with $\gConf{c'}{\ell'}$ as origin\;
			 }
% 		${\mathcal G}''$ the subgraph of \G{i} induced by the subgraph $c'_{B_{q}^{(i)}} \leftrightarrow c'_{B_{q+1}^{(i)}} \leftrightarrow \ldots \leftrightarrow c'_{B_{p-1}^{(i)}} \leftrightarrow \gConf{c}{\ell}$\;
% 		\If{$c$ is compatible with $c'$ }{
% 			 Merge ${\mathcal G}$ with the graph ${\mathcal G}''_{q-1} \times \gConf{c'}{\ell'} \times {\mathcal G'_{q+1 }}$ at levels $q-1, q, q+1$.
		}
	}
}
\If{level $q$ of ${\mathcal G}$ is empty}{
	\Return {\tt false}\;
}{
	\For{$\gConf{c'}{\ell'}$ in level $q$ of ${\mathcal G}$}{
	Add all edges from $\gConf{c'}{\ell'}$ to each vertex of level $q+1$ of ${\mathcal G}$;
	}
}
\Return ${\mathcal G}$\;
}
\caption{$iteratepGreaterThanq(\G{i})$ \label{algo:stepp>q}}
\end{algorithm}

This case is very similar to the preceding one except that $B_{p}^{(i)}$ is not cut into pieces but glued together with preceding blocks. 
As a consequence, when testing accessibility of a configuration of $B_{q}^{(i+1)}$, 
we should consider every corresponding configuration in \G{i}, 
that is every configuration obtained by stacking configurations at level $q,q+1,\ldots, p$ in \G{i}. 
Unfortunately this may give an exponential number of configurations, 
but noticing that by Lemma \ref{lem:HautGaucheA} elements of blocks $B_{q}^{(i)},B_{q+1}^{(i)} \ldots B_{p-1}^{(i)}$ 
are exactly in the same stack at time $t_{i}$ and at time $t_{i+1}$, 
it is sufficient to check the accessibility of a pushall configuration $c'$ of $B_{q}^{(i+1)}$ from a configuration $c$ of $B_{p}^{(i)}$ 
and verify afterwards whether the configuration $c$ has ancestors in \G{i} that match exactly the configuration $c'$. 
This leads to the Algorithm~\ref{algo:stepp>q}.
%
%A more precise algorithm is given hereafter to build \G{i+1} from \G{i}. 
%Let $G'$ be the sorting graph of $\sigma^{(i+1)}$ that is obtained using \cite{PR13}. 
%First, we copy the $q-1$ first levels of \G{i} into \G{i+1}.
%Then for every configuration $(c',m)$ at level $q$ in $G'$ corresponding to the block $B^{(i+1)}_{q}$ and every configuration $(c,\ell)$ at level $p$ in \G{i} corresponding to block $B_{p}^{(i)}$, if $c$ and $c'$ are compatible then:
%
%\begin{itemize}
%\item Add $(c',m)$ in level $q$ of \G{i+1} if not exists
%\item If configurations $c'_{|B_{q}}, c'_{|B_{q+1}},\ldots c'_{|B_{p-1}} (c,\ell)$ form a path in \G{i} then denoting $(c'',h)$ the corresponding permutation at level $q$ of \G{i}:
%\begin{itemize}
%\item Add edges between $(c',m)$ and the configurations at level $q-1$ of \G{i} that are adjacent to $(c'',h)$.
%\end{itemize}
%\end{itemize}

%At the end complete \G{i+1} by copying levels $q+1,q+2, \ldots$ of $G'$ and add all possible edges between level $q+1$ and level $q$.

Note that in Algorithm~\ref{algo:stepp>q}, before calling $isAccessible(c,c',\sigma_{|B^{(i)}_{p} \bigcup B^{(i+1)}_{q}})$
we extend configuration $c$ to $D^{(i)} \bigcup B_{p}^{(i)}$ 
by assigning the same stack than in $c'$ to points of $D^{(i)} \setminus B_{p}^{(i)}$.
This is justified by Lemma~\ref{lem:HautGaucheA}.

Now that we have described all steps of our algorithm, we turn to the study of its complexity.

\section{Complexity Analysis}

In this section we study the complexity of our main algorithm: {\em isSortable($\sigma$)} (Algorithm~\ref{algo:total}).
%
% \begin{lem}
% Lemma \ref{lem:compatible} can be verified in linear time.
% \end{lem}
% 
% \begin{proof}
% Lemma \ref{lem:compatible} relies on Lemma \ref{lem:compatibleRectangle} which runs in linear time. 
% In Lemma \ref{lem:compatible}, some elements are assigned to stacks before invoking  Lemma \ref{lem:compatibleRectangle}. 
% These assignments are easily done in linear time.
% \end{proof}
% 
% The next step is to study the complexity of procedure {\em SORTABILITY}. 
% This procedure involves three other subroutines. 
The key idea for the complexity study relies on a bound of the size of each graph \G{i}, as described in the following lemma.

\begin{lem}\label{lem:boundLevel}
For any $i \in [1 .. r]$, the maximal number of vertices in a level of \G{i} is $9n+2$
where $n$ is the size of the input permutation.
\end{lem}

\begin{proof}
From Theorem 4.4 of \cite{PR13}, 
the maximal number of pushall stack configurations of a $\ominus$-indecomposable permutation $\pi$ is $9|\pi|+2$. 

By definition of \G{1}, the vertices of a level correspond to 
pushall stack configurations of a given block of the $\oplus_{1}$-decomposition of the input permutation $\sigma$.
Thus the cardinality of a level is bounded by $9k+2$ where $k$ is the size of the corresponding block.
As $k \leq n$, the result holds for $i = 1$ . 

Suppose now that the result is true for a given \G{i}, we show that it is then true for \G{i+1}.
The graph \G{i+1} is build from \G{i} using Algorithm~\ref{algo:steppq}, \ref{algo:stepp<q} or \ref{algo:stepp>q}.
In each case for a level $j$ of \G{i+1} we have:

If $j> q^{(i+1)}$ then vertices of the level $j$ of \G{i+1} are the pushall stack configurations 
corresponding to the block $B_j^{(i+1)}$ of the $\oplus_{i+1}$-decomposition of $\sigma$.
Thus Theorem 4.4 of \cite{PR13} ensures that the cardinality of level $j$ is bounded by $9n+2$.

If $j = q^{(i+1)}$ then vertices of the level $j$ of \G{i+1} are a subset of the pushall stack configurations 
corresponding to the block $B_j^{(i+1)}$ of the $\oplus_{i+1}$-decomposition of $\sigma$.
Again Theorem 4.4 of \cite{PR13} ensures that the cardinality of level $j$ is bounded by $9n+2$.

If $j < p^{(i)}$ then vertices of the level $j$ of \G{i+1} are a subset of vertices of the level $j$ of \G{i}.
By induction hypothesis the cardinality of level $j$ is bounded by $9n+2$.

If $p^{(i)} \leq j < q^{(i+1)}$ then vertices of the level $j$ of \G{i+1} are restrictions of a subset of vertices of the level $j$ of \G{i}.
By induction hypothesis the cardinality of level $j$ is bounded by $9n+2$, concluding the proof.
\end{proof}

\begin{lem}\label{lem:boundGraph}
For any $i \in [1 .. r]$, the number of vertices of \G{i} is ${\mathcal O}(n^2)$
and the number of edges of \G{i} is ${\mathcal O}(n^3)$,
where $n$ is the size of the input permutation.
\end{lem}

\begin{proof}
The result follows from Lemma~\ref{lem:boundLevel} as there are at most $n$ levels and there are edges only between consecutives levels. 
\end{proof}

\begin{thm}
Given a permutation $\sigma$, Algorithm \ref{algo:total} {\em isSortable($\sigma$)} decides 
whether $\sigma$ is sortable with two stacks in series in polynomial time w.r.t.~$|\sigma|$.
\end{thm}

\begin{proof}
Algorithm \ref{algo:total} involves four other subroutines: 
{\em ComputeG1} (Algorithm~\ref{algo:step1}), {\em iteratepEqualsq} (Algorithm~\ref{algo:steppq}), 
{\em iteratepLessThanq} (Algorithm~\ref{algo:stepp<q}) and {\em iteratepGreaterThanq} (Algorithm~\ref{algo:stepp>q}). 

Each for-loop in these algorithms is executed at most a linear number of time by Lemma \ref{lem:boundLevel}.

Moreover each included operation is polynomial by Lemmas~\ref{lem:boundGraph} and \ref{lem:compatibleRectangle}.
% 
% Indeed, the only non-trivial operations are:
% \begin{itemize}
% \item Merging graph ${\mathcal G}$ and ${\mathcal G'}$ : 
% In this operation we have to check whether the vertices of ${\mathcal G'}$ exist in ${\mathcal G}$ and create them if needed. 
% Looking for a vertex is done in linear time. Thus this merge operation could be performed in time $|{\mathcal G}| \times |{\mathcal G'}|$.
% \item Extract the  graded subgraph of ${\mathcal G}$ induced by the graded graph ${\mathcal G'}$. 
% This time, the vertices of ${\mathcal G'}$ are known and the extraction can easily be done by a traversal of the graph which is done in linear time with respect to $|{\mathcal G}|$.
% \end{itemize}
\end{proof}

Notice that a more precise analysis of complexity leads to an overall complexity of ${\mathcal O}(n^{5})$.

\bibliographystyle{habbrv}
\bibliography{biblio}
\end{document}